\documentclass[12pt,web,onecolumn]{ieeecolor}
\usepackage{generic}
\usepackage[cmyk]{xcolor}

\usepackage{amsmath,amssymb,amsfonts,amsthm}
\usepackage{pgfplots}

\usepackage{mathrsfs}
\usepackage{setspace} 
\usepackage{arydshln}
\usepackage[framemethod=tikz]{mdframed}
\usepackage{multicol}
\usepackage{hyperref}

\usepackage{paralist}

\usepackage{subcaption}
\usepackage{graphicx}                        
\usepackage{color}                           

\usepackage{tikz}
\usepackage{rotating} 
\usetikzlibrary{positioning, shapes.geometric}
\usetikzlibrary{graphs}
\usepackage{cite}

\usepackage{enumerate}

\usetikzlibrary{shapes, arrows}
\usepackage{pdflscape}

\usepackage{todonotes}
\usepackage{silence}
\DeclareMathOperator{\diag}{diag}
\providecommand\inva[1]{\,\mathsf{d}#1}
\DeclareMathOperator{\He}{He}
\DeclareMathOperator{\tr}{tr}

\theoremstyle{plain}\newtheorem{theorem}{Theorem}
\theoremstyle{plain}\newtheorem{lemma}{Lemma}
\theoremstyle{plain}
\theoremstyle{plain}\newtheorem{proposition}{Proposition}
\theoremstyle{definition}
\theoremstyle{definition}\newtheorem{remark}{Remark}
\theoremstyle{definition}\newtheorem{definition}{Definition}
\theoremstyle{definition}\newtheorem{example}{Example}
\theoremstyle{definition}\newtheorem{assumption}{Assumption}

\def\jmo{\jmath\omega}
\def\jinf{\jmath\infty}
\def\Ltfd{\mathcal L_2(-\jinf, \jinf)}
\def\Lttd{\mathcal L_2[0, \infty)}
\def\Htff{\mathbb U(\Omega_f)}

\def\BibTeX{{\rm B\kern0.05em{\sc i\kern0.025em b}\kern0.08em
		T\kern-18.1667em\lower17.17ex\hbox{E}\kern 0.125emX}}
\begin{document}
	\pgfplotsset{compat=1.14}
	
\title{Extension of generalized KYP lemma: from LTI systems to LPV systems}
	\author{Jingjing Zhang, Jan Heiland, Peter Benner, Xin Du 
 \vspace{-2.5em}
		\thanks{Manuscript received xx-yy, 20zz; revised xx-yy, 20zz. }
		\thanks{Jingjing Zhang is with the School of Mechatronic Engineering and Automation, Shanghai University, Shanghai, 200444, China,  and also with the State Key Laboratory of Mathematical Sciences, Academy of Mathematics and Systems Science, Chinese Academy of Sciences, 100190, China (e-mail: zhang-jingjing@shu.edu.cn) }
		\thanks{Jan Heiland is with the Max Planck Institute for Dynamics of Complex Technical Systems, 39106 Magdeburg, Germany (e-mail: heiland@mpi-magdeburg.mpg.de)}
         \thanks{Peter Benner is with the Max Planck Institute for Dynamics of Complex Technical Systems, 39106 Magdeburg, Germany (e-mail: benner@mpi-magdeburg.mpg.de).}
		\thanks{Xin Du was with the School of Mechatronic Engineering and Automation, Shanghai University, Shanghai, 200444, China, before April, 2023. His contribution for this work was conducted as an independent researcher. Corresponding author. (e-mail: duxin@shu.edu.cn)}}

	\maketitle

\begin{abstract}

The generalized Kalman-Yakubovich-Popov (gKYP)
lemma, established by Iwasaki and Hara (2005 IEEE TAC), has served as a fundamental tool for finite-frequency analysis and synthesis of linear time-invariant (LTI) systems. 
Over the past two decades, efforts to extend the gKYP lemma from LTI systems to linear parameter varying (LPV) systems have been hindered by the intricate time-frequency inter-modulation effect between the input signal and the time-varying scheduling parameter.
A key element in this framework is the frequency-dependent Integral Quadratic Constraint (IQC) function, which enables 
time-domain interpretation of the gKYP lemma, as demonstrated by Iwasaki et al in their companion 2005 System and Control Letter paper. 
The non-negativity property of this IQC function plays a crucial role in characterizing system behavior under frequency-limited inputs.
In this paper, we first demonstrate through a counterexample that the IQC non-negativity property may fail for LPV systems, thereby invalidating existing results that rely on this assumption. 
To address this issue, we propose a reformulation strategy that replaces the original frequency range with an enlarged one, thereby restoring the non-negativity property for LPV systems.
Moreover, we establish that the minimal required expansion depends on the interaction(or gap) between the system poles and the original frequency range, as well as a set of controllability Gramians. 
Building upon this results, an extension of gKYP lemma is presented, which allows us to conduct finite-frequency analysis of LPV systems in a direct and reliable manner. The potential and efficiency compared to existing results are demonstrated through numerical examples.

\end{abstract}


\section{Introduction}
The input-output performance plays a crucial role in the analysis and synthesis
of dynamical systems \cite{IO_oppenheim1997signals}. 
With the input signal being subjected to the user's choice or technical
limitations,
one may well assume that
it's spectrum 
is restricted to an
{\it a priori} known frequency range. 
Similarly naturally, any
input-output performance analysis over a finite frequency range rather than over
the
entire frequency range, can lead to improved results, as it can facilitate
given information of the relevant spectrum
distribution.


For linear time-invariant (LTI) systems, as the most standard model, the
input-output performance had been well addressed both on the entire frequency
range and on different finite frequency ranges. It is well-known that various
performance analyses and control design problems defined on the entire frequency range can be solved on the basis of the classic Kalman-Yakubovich-Popov (KYP)
lemma which is attributed to the original works by Yakubovich \cite{KYP_Yakubovich} from 1962,
Kalman \cite{KYP_Kalman} from 1963 and Popov \cite{KYP_Popov} from 1964. 
In the mid-90s, several works (see, e.g., \cite{KYP_Rantzer1996on}), 
have used the KYP lemma to establish an equivalence between performance
conditions in the (entire) frequency domain, input–output relationship of the system in the time domain, and conditions on the matrices describing the state space representation of the system. 
Because of it's fundamental nature, the KYP lemma has become recognized as a cornerstone of modern control theory. 
Thanks to the contribution of Iwasaki and Hara (see, e.g.,
\cite{gKYP_iwasaki2005generalized} and \cite{gKYP_iwasaki2005time}), the system
and control research was equipped with a generalized KYP (gKYP) lemma, which is
an extension to deal with performance analysis and control design problems
defined on a finite-frequency range. The last 20 years witnessed the success of
the gKYP lemma, with more than 1200 works that cite the gKYP lemma to now cover,
e.g., fault detection and controller and filter design for different engineering
applications.

Linear parameter varying (LPV) modeling is a widely adopted framework in control
systems that may overcome some limitations of LTI systems while keeping some of the
well-developed linear theory and techniques in place. 
Through the construction of parameter-dependent control Lyapunov functions,
valid extensions of the KYP lemma to LPV systems
have been reported on the entire frequency range; see
\cite{LPV_EF_wu1996induced,LPV_EF_polcz2020induced}. However, for
finite-frequency analysis/synthesis problems, no extension of the gKYP lemma for
LPV systems has been developed yet. A major issue is that, different from the
LTI case, the output signal of a general dynamical system will not necessarily
be in band with a possibly frequency-limited input signal. Correspondingly, the
input-output relationship cannot be expressed in terms of a transfer function or
other kinds of frequency domain inequality (FDI). 

Several attempts have been made to expand the gKYP lemma from LTI systems to LPV
systems; see, e.g., \cite{S_ding2010fuzzy,S_sun2013induced, S_Wang2017H,
S_baar2020parameter, S_li2014fault}. A remarkable common feature of these works
is the replacement of the finite-frequency input signal assumption by an
frequency-dependent {\it Integral Quadratic Constraint} (IQC) assumption. Such
IQCs were introduced by Iwasaki and Hara in their often underappreciated but
important companion paper \cite{gKYP_iwasaki2005time} to provide a time-domain
interpretation of the gKYP lemma. By distorting the fundamental assumption,
indeed, sufficient linear matrix inequality (LMI) conditions with a similar
structure of the LMI conditions in the gKYP lemma were derived, and the
numerical examples therein show that one may obtain enhanced in-band
analysis/control design performance by adopting the conditions. Nevertheless,
those attempts are hardly to be recognized as successful extension of the gKYP
lemma, as they suffer from the following intrinsic shortcomings:

  \begin{itemize}
\item[$\bullet$]  
For the sake of practical applicability, the assumption that the input signal is
frequency limited makes sense as the critical frequency-band information can be
obtained in {\it a priori}. On the contrary, the IQC is defined with regard to the
system's state and its derivative so that the frequency-dependent IQC can only
be verified {\it a posteriori}. In other words, taking the frequency-dependent
IQC as the fundamental assumption does not well apply unless the IQC assumption
can be established together with the finite-frequency assumption on the input signal. 

\item[$\bullet$] 
In an LTI system setting, although it is not stated explicitly in
\cite{gKYP_iwasaki2005time}, a finite-frequency input signal will always render the system state and its derivative satisfying the frequency-dependent IQC.
In the LPV system setting, the relationship between the finite-frequency input
signal and the frequency-dependent IQC is not well explored. Even more, the
unavoidable occurrence of out-of-band frequency components in the system makes
it difficult to a-priori infer frequency-dependent IQCs based only on
finite-frequency input signals. We will provide the direct relation of finite
frequencies and the IQC for the LTI case
below and a counterexample for the LPV case.
\end{itemize}



In this paper, we shed some light on the correlation of finite-frequency input
signals and the IQC for the system's state and its derivative therewith
developing an applicable and well-founded counterpart of the gKYP lemma for the LPV system setting. Specifically speaking, our contributions are:


  \begin{itemize}
\item[(1)] 
Starting with the LTI case, we first re-state how the finite-frequency assumption implies a
finite-frequency IQC with respect to system state and its derivative to make it
more explicit. At the same time, the physical meaning of the implication is
re-examined, while we further reveal how the concept of controllability (as well
as the controllability Gramians) plays an important role in this relation.

\item[(2)] 
For the LPV case, we construct a counterexample that illustrates how IQC may not
hold despite a finite-frequency assumption on the input signals. This
counterexample also addresses the shortcomings of the previous attempts.

\item[(3)] 
By re-interpreting the structure of a finite-frequency IQC, we show that the
sufficiency from finite-frequency input signals to some IQC condition can be
recovered with an frequency-range enlargement scheme. We show that there always
exist a sufficiently large frequency range (including the pre-specified
frequency range as a subset) so that an IQC will eventually hold. The minimal
necessary enlargement of the frequency range is derived and presented as a
simple and insightful formula that includes the gap between the system pole
allocation and the given frequency range in the complex plane, and traces of a
family of controllability Gramians.

\item[(4)] 
Based upon the re-constructed IQC with regard to the enlarged frequency range, a
valid extension of gKYP lemma for LPV system is presented, which, in addition,
enables a direct and reliable finite-frequency analysis of LPV systems. The
potential and efficiency compared to existing results is demonstrated through a numerical example.

\end{itemize}

Notations: For a matrix $M$, the transpose and complex conjugate transpose are denoted by $M^{T}$ and $M^{*}$, respectively. $\He(M)$ represents $M^*+M$.
We use $M<(>)0$ and $M\leq(\geq) 0$ to indicate that the matrix $M$ is negative (positive) and semi-negative (semi-positive) definite, respectively.  
And $\tr(M)$ is the trace of matrix $M$. 
In view of matrices, $I$ and $0$ represent identity matrix and zero matrix with appropriate dimensions. By $\diag\{\ldots\}$, we represent a block-diagonal matrix with the entries to be specified. 
We let $\mathbb{R}^+$, $\mathbb{R}^n$ and $\mathbb{R}^{n\times m}$ denote the set of positive real numbers, $n$-dimensional Euclidean space and the vector space of $n\times m$ matrices with real entries. 
$\mathbb{S}^n$ represents the set of $n\times n$ real symmetric matrices. 
The symbol $\lambda(\cdot)$, $\lambda_{\max}(\cdot)$ and $\lambda_{\min}(\cdot)$
and 
denote an eigenvalue of a general matrix, and
the largest and smallest eigenvalue of a symmetric matrix.

\section{Revisit of gKYP lemma}\label{sec:LTI}

\subsection{LTI systems, basics, KYP lemma}

In this section, we revisit the well-established finite-frequency analysis 
of LTI systems and clarify the conceptual ambiguities arising from frequency
domain/time domain characterization of the signals and performance index.
Consider a continuous-time LTI models in state-space form
\begin{equation}\label{LTI-time}
	\begin{split}
		\dot{x}(t) &= A x(t)+B  u(t),\\
		y(t) &= C x(t)+D u(t),
	\end{split}
\end{equation}
where $u\in \mathbb R^{p}$ is the input signal, $x\in \mathbb R^{n}$ is the
system state, and $y\in \mathbb R^{q}$ is the output signal.
In what follows, we rely on the equivalence of square-integrable signals 
$u\in \mathcal L_2[0,\infty)$ 
and $\mathcal{U}\in \mathcal L_2(-\jmath \infty, \jmath \infty)$
in time and frequency domain, as it is induced by the Fourier transform
$\mathcal F$ defined as 
\[\begin{array}{l}
   \mathcal{F}(u)(\jmo) := \mathcal{U}(\jmath \omega)= \int_{-\infty}^\infty u(t) e^{-\jmath \omega t}   \inva t,  
\end{array}\]
where $\jmath$ is the imaginary unit, and its inverse
\[\begin{array}{l}
  \mathcal{F}^{-1}(\mathcal{U})(t) : = u(t) = \frac{1}{2\pi} \int_{-\infty}^\infty \mathcal{U}(\jmath \omega) e^{\jmath \omega t}\inva\omega,
\end{array}\]
for a frequency $\omega\in \mathbb R$ or time $t\in \mathbb R$ together
with the \emph{Parseval identity} that states that
\begin{equation}\label{eq:Parseval}
\begin{array}{l}
  \int_{-\infty}^{\infty} u^*(t)u(t) \inva t := \|u\|^2_{\mathcal L_2(-\infty,
  \infty)} = \|\mathcal{U}\|^2_{\mathcal L_2(-\jmath\infty, \jmath\infty)} :=
\frac{1}{2\pi} \int_{-\infty}^{\infty} \mathcal{U}^*(\jmath \omega)\mathcal{U}(\jmath \omega) \inva \omega.
\end{array}
\end{equation}

For signals $u \in \mathcal L_2[0, \infty)$ on the half-line, the formulas
readily apply and, assuming that $x(0)=0$, we can consider the frequency domain
representation of the LTI model \eqref{LTI-time} as  
\begin{equation}\label{LTI-Fre} 
	\begin{split}
		\jmath \omega \mathcal X(\jmath \omega) &= A \mathcal X(\jmath\omega)+B  \mathcal U(\jmath\omega), \\
		 \mathcal Y(\jmath\omega) &= C \mathcal X(\jmath\omega) + D \mathcal U(\jmath\omega),
	\end{split}
\end{equation}
that defines the transfer function matrix $G\colon (-\jmath \infty, \jmath
\infty) \to \mathbb R^{q\times p}$ by means of
\[\begin{array}{l}
     \mathcal Y(\jmath\omega) =[C(\jmath\omega I-A)^{-1}B+D] \mathcal U(\jmath\omega) :=G(\jmath\omega) \mathcal U(\jmath\omega).
\end{array}\]

We note that the Fourier transformation identifies the space $\Lttd$ of square
integrable signals on the half-line $[0, \infty)$ in time domain with the $\mathcal H_2$ (\emph{Hardy}-)space of signals that are
analytic in the open right-half plane in the frequency domain; see, e.g.,
\cite[Sec. 4.3]{ZhoDG96}.

The central topic in system analysis is the characterization of input-output
relationships. As, typically, norms of the signals are considered, many
input-output relationship can be described in a quadratic integral form like in the general so-called input-output performance in time domain, i.e., 
\begin{equation}\label{IO-Performance-TD}
\int_{0}^{\infty} \begin{bmatrix}
y(t)\\
u(t) 
\end{bmatrix}^* \Pi
\begin{bmatrix}
y(t)\\
u(t) 
\end{bmatrix}  \inva t \leq 0, 
\end{equation}
where the Hermitian matrix $\Pi \in  \mathbb R^{(q+p)\times (q+p)}$ presents the performance index matrix,
or (as an application of the Parseval identity \eqref{eq:Parseval}; cp., e.g.,
\cite[Th. 1]{gKYP_iwasaki2005time}), the equivalent input-output performance in frequency domain, i.e., 
\begin{equation}
\label{IO-Performance-FD}
\int_{-\infty}^{\infty} {\begin{bmatrix}
\mathcal Y(\jmath\omega)\\
\mathcal U(\jmath\omega) 
\end{bmatrix} }^* \Pi{\begin{bmatrix}
\mathcal Y(\jmath\omega)\\
\mathcal U(\jmath\omega) 
\end{bmatrix} } \inva\omega \leq 0.
\end{equation}

Note that with $\Pi=\diag\{I,-\gamma^2 I\}$, the input-output relationship recasts to the well-known induced $\mathcal L_2$ gain in
time domain (or the $\mathcal H_\infty$ norm in frequency domain); see
\cite{IO_wollnack2017fixed}.

The inequality \eqref{IO-Performance-FD} is readily formulated with regard to the finite-frequency property of the transfer function $G(\jmath\omega)$ as follows:
\begin{equation}
\begin{bmatrix}
   G(\jmath\omega)\\
I  
\end{bmatrix}^*\Pi
\begin{bmatrix}
   G(\jmath\omega)\\
I  
\end{bmatrix} 
\leq 0, \quad \forall \omega \in (-\infty, +\infty).
\end{equation}

In the 1960s, with the famous KYP lemma, a sufficient and necessary as well as numerically tractable solution in terms of LMI had been established for analyzing input-output
performance indices. 
Let us recall the KYP lemma now.


\begin{lemma}[KYP lemma, \cite{KYP_Rantzer1996on}] \label{lem:KYP}
Consider the LTI system \eqref{LTI-time} with $(A,B)$ controllable, and a Hermitian matrix $\Pi \in  \mathbb R^{(q+p)\times (q+p)}$. For any input signal $u\in\mathcal{L}_2[0,\infty)$, the following statements are equivalent:

\begin{enumerate}[(1)]
  \item
There exists a Hermitian matrix $P$, such that
 \begin{equation} \label{KYP_LMI}
\begin{bmatrix}
				A &  B\\
				I& 0 
 \end{bmatrix} ^*
              (\Theta \otimes P)    
              \begin{bmatrix}
				A &  B \\
				I & 0
              \end{bmatrix}
      +
      \begin{bmatrix}
                C & D  \\
				0 & I
			\end{bmatrix}^* \Pi 
      \begin{bmatrix}
				C    & D   \\
				0 & I 
			\end{bmatrix}   \leq 0, ~{\rm with}~\Theta=\begin{bmatrix}
    0& 1\\
    1 & 0
\end{bmatrix}.
	\end{equation}

  \item
The input-output performances in time- and frequency domain,
\eqref{IO-Performance-TD} and \eqref{IO-Performance-FD}, hold.

\end{enumerate}

\end{lemma}

\subsection{Finite-frequency input signal, gKYP lemma, version 1}

In what follows, we will consider the subspace
of frequency-limited (with respect to $\Omega_f$) functions denoted by
\begin{equation}\label{eq:H2-finite-freq}
  \Htff:=\{\mathcal{U} \in \mathcal H_2: \mathcal{U}(\jmo) = 0, \,\text{for }\omega \notin \Omega_f\}.
\end{equation}
In particular, for $\mathcal{U}\in \Htff$ and $u:=\mathcal F^{-1}(\mathcal{U})$, we have that, $u(t)=0$, for
$t\leq 0$ so that $\mathcal{U}(\jmo) = \mathcal F(u) = \int_0^\infty u(t)e^{-\jmo
t}\inva t$.

Furthermore, one commonly distinguishes distinctive frequency ranges $\Omega_f$ as 
low-frequency range $\Omega_l$, middle-frequency range $\Omega_m$, or
high-frequency range $\Omega_h$ by means of
threshold values $\varpi_l$, $\varpi_1 < \varpi_2$, $\overline
\omega_h$ and the requirements that
\begin{subequations}
\begin{align}
&  \Omega_l:= [-\varpi_l,\varpi_l],\\
&\Omega_m := [-\varpi_2,
-\varpi_1]\cup [\varpi_1, \varpi_2], \\
& \Omega_h :=
(-\infty,-\varpi_h]\cup[\varpi_h,\infty).
\end{align}
\end{subequations}

Moreover, it is helpful to allow for the consistent notion that $\Omega_f=(-\infty, \infty):=\Omega_e$, when the signal has support on the entire-frequency range.

We note that the finiteness of the signals can be similarly characterized by the
inequality that for all $\omega \in \Omega_f$ one has
\begin{equation} \label{eq:f-positive-inband}
  f(\omega,\Omega_f)=\begin{bmatrix}
    \jmath\omega\\
    1
\end{bmatrix}^*  \Psi_f \begin{bmatrix}
    \jmath\omega\\
    1
\end{bmatrix} \geq 0,
\end{equation}
where $\Psi_f$ is a characteristic matrix generated for the frequency range that
can be either a segment of a straight line in the complex plane, to particular
curves in the complex plane described by a polynomial equality or a polynomial
inequality; see \cite{Pipeleers14}. The considered set of curves is known to
include the union of segments of a line as a special case with the choices of
\begin{subequations}\label{eq:psif}
\begin{align}
&\Psi_f=
\left[\begin{smallmatrix}
 -1 & 0\\
 0 & \varpi_l^2
\end{smallmatrix}\right],{\kern 30pt}\Omega_f \subset \Omega_l,\\
& \Psi_f=\left[\begin{smallmatrix}
 -1 & \jmath \varpi_c\\
 -\jmath \varpi_c & -\varpi_1\varpi_2
\end{smallmatrix}\right],~ \Omega_f \subset \Omega_m,\\
& \Psi_f=\left[\begin{smallmatrix}
 1 & 0\\
 0 & -\varpi_h^2
\end{smallmatrix}\right], {\kern 30pt} \Omega_f \subset \Omega_h,
  \end{align}
\end{subequations}
for the thresholds $\varpi_l$, $\varpi_1<\varpi_c=\frac{\varpi_1+\varpi_2}{2}<\varpi_2$, or $\varpi_h\in \mathbb
R$.
It has been noted that the validity of the input-output performance inequality \eqref{IO-Performance-TD}/\eqref{IO-Performance-FD} greatly depends on the
class of input signals considered so that specific versions have been developed
for advanced system analysis and synthesis. 
In real-world systems, the
input signals may only range in a certain frequency band (in the simplest case
they are composed of a sum of sinusoids and, thus, have a
finite and discrete spectrum) so that many relevant results have been found
under a finite-frequency spectrum assumption
\cite{disturbance_zhang2023FDSC,disturbance_marino2016hybrid}. In the case that
the spectrum of the input signals is within a finite-frequency range $\Omega_f$, the inequality \eqref{IO-Performance-FD} can be equivalently rewritten as finite-frequency domain form
\begin{equation}\label{IO-Performance-FD-FF}
\int_{\Omega_f}
{\begin{bmatrix}
\mathcal Y(\jmath\omega)\\
\mathcal U(\jmath\omega) 
\end{bmatrix} }^* \Pi \begin{bmatrix}
\mathcal Y(\jmath\omega)\\
\mathcal U(\jmath\omega) 
\end{bmatrix}  \inva\omega \leq 0, 
\end{equation}
and in finite-frequency assumption as 
\begin{equation}\label{gKYP_equ2}
\begin{bmatrix}
 G(\jmath\omega)\\
I   
\end{bmatrix}^*\Pi
\begin{bmatrix}
 G(\jmath\omega)\\
I   
\end{bmatrix}
 \leq 0, ~\forall \omega \in \Omega_f.
\end{equation}

Indeed, by pre-multiplying
$\mathcal U^*(\jmath\omega)$ and post-multiplying  $\mathcal U(\jmath\omega)$ in
\eqref{gKYP_equ2}, we obtain the constrained input-output performance
\eqref{IO-Performance-FD} with respect to finite-frequency $\Omega_f$,
where in both cases, $\Pi\in \mathbb R^{(q+p)\times (q+p)}$ is the
performance index matrix. 
Under the finite-frequency input signal assumption, the desirable input-output performance can be refined in a finite-frequency input specific form.



\begin{assumption}[Input signal with finite-frequency spectrum]\label{def:ff-signals}
  Throughout this paper, the input signals have a limited spectrum. As for the
  notation, 
  a signal $\mathcal{U}\in \Ltfd$ will be
  referred to as a finite-frequency signal with respect to a frequency range
  $\Omega_{f}\subset (-\infty, \infty)$, if 
  \begin{equation}\label{FD-FFS-strict}
    \mathcal{U}(\jmo)=0, ~ \text{for }\omega \notin \Omega_f.
  \end{equation}
\end{assumption}


Since finite-frequency properties carry significant physical meanings, they are often verifiable, and critical frequency parameters (such as the upper and lower bounds) are typically available.

\begin{lemma}[gKYP lemma, \cite{gKYP_iwasaki2005generalized}] 
\label{lem:gKYP-FD}
Consider the LTI system \eqref{LTI-time} with $(A,B)$ controllable, and a finite-frequency range $\Omega_f$ 
with the corresponding weight matrix $\Psi_f$ as defined in \eqref{eq:psif}. Suppose $\Pi \in  \mathbb R^{(q+p)\times (q+p)}$ is Hermitian, and the frequency spectrum of input signal $u\in\mathcal{L}_2$ is limited to finite range $\Omega_f$. 
Then, the following statements are equivalent:
\begin{enumerate}[(1)]
  \item
There exist Hermitian matrices $P$, $Q\geq0$, such that
 \begin{equation} \label{gKYP_LMI}
\begin{bmatrix}
				A &  B\\
				I& 0 
 \end{bmatrix} ^*
              (\Theta \otimes P +\Psi_f \otimes Q)    
              \begin{bmatrix}
				A &  B \\
				I & 0
              \end{bmatrix}
      +
      \begin{bmatrix}
                C & D  \\
				0 & I
			\end{bmatrix}^* \Pi 
      \begin{bmatrix}
				C    & D   \\
				0 & I 
			\end{bmatrix}   \leq 0,~{\rm with}~\Theta=\begin{bmatrix}
    0&1\\
    1&0
\end{bmatrix}.
	\end{equation}

  \item
The input-output performance \eqref{IO-Performance-FD} holds.

\end{enumerate}

\end{lemma}

We note that the KYP lemma does not facilitate {\it a priori} frequency conditions on the
signals, so it was the fundamental work on the gKYP lemma in the 2000s (see, e.g.,
\cite{gKYP_iwasaki2005generalized}) that triggered 
various and promising works on finite-frequency oriented analysis and synthesis problems for linear systems \cite{Cai2024H,Sun2020advanced}. 

Since the gKYP lemma (also in its time domain version \cite{gKYP_iwasaki2005time})
is key to many works in the last two decades as well as to our approach, we give a complete recall.

\subsection{Time domain interpretation of finite-frequency input signal, gKYP lemma, version 2}

Consider the LTI system \eqref{LTI-time} with an initial state $x_0=x(t_0)$, the unique solution of the system has the representation
\begin{equation}
    x(t)=e^{A(t-t_0)}x_0+\int_0^t e^{A (t-\tau)} B u(\tau) \inva\tau,~t\geq t_0.
\end{equation}
Then for the initial state is zero and a finite-frequency input signal, i.e.,
\begin{equation}\label{eq:TD-IP}
\begin{array}{l}
		u(t)= 
        \int_{\Omega_f} \mathcal U(\jmath\omega)e^{\jmath\omega t} \inva\omega,
\end{array}
\end{equation}
the state vector and its derivative vector are given by
\begin{equation}\label{x-dx-LTI}
\begin{array}{l}
	x(t)
   = \int_{\Omega_f} (\mathcal{V}_1(\jmath\omega) +\mathcal{V}_2(\jmath\omega)) \mathcal U(\jmath\omega) \inva\omega,
    \\
\dot x(t)
=\int_{\Omega_f}
    (\jmath\omega \mathcal{V}_1(\jmath\omega) + A \mathcal{V}_2(\jmath\omega)) \mathcal U(\jmath\omega) \inva\omega,
\end{array}
\end{equation}
with
\begin{equation}\label{V-LTI}\begin{array}{l}
   \mathcal{V}_1(\jmath\omega) =   
   e^{\jmath\omega t} (\jmath\omega I-A)^{-1} B,
   \\
   \mathcal{V}_2(\jmath\omega) 
   =-e^{At} (\jmath\omega I-A)^{-1} B.
\end{array}
\end{equation}






In the following, we will argue the non-negativity of $\Omega_f$-IQC rendering signals for general LTI systems. For that, we recall the concept of state transition matrices $e^{At}$ of LTI system as well as its associated uniform boundedness properties as discussed in \cite{gKYP_iwasaki2005time}.

\begin{lemma}[Boundedness properties of LTI
  systems, \cite{gKYP_iwasaki2005time}] 
  \label{UB_LTI}
Consider the LTI system \eqref{LTI-time} and suppose that $A$ is Hurwitz. Then
for any given square matrix $M$ and all $t\geq 0$, we have 
\begin{enumerate}[(1)]
  \item
$\| e^{A t} \| \leq \sqrt{{\rm cond}(P)}$,
\item 		
$\|\int_{0}^{t} e^{A \tau} \inva\tau \|
		\leq  \| A^{-1} \| (1+\sqrt{{\rm cond}(P)})$,
\item
  $\| \int_{0}^{t} e^{A^* \tau} M  e^{A \tau} \inva\tau \|
		\leq \| P \| \frac{\| M+M^*\|+\| M-M^*\|}{2}$,
\end{enumerate}  
where $P:=\int_0^\infty e^{A^* t} e^{At} \inva t$, and ${\rm
cond}(P)=\frac{\lambda_{\max}(\int_{0}^{\infty} e^{A^* t} M  e^{A t} \inva t)
}{\lambda_{\min}(\int_{0}^{\infty} e^{A^* t} M  e^{A t} \inva t) }$, i.e., the
fraction of the largest versus the smallest eigenvalue of the symmetric positive definite Gramian.
\end{lemma}

Inspired by the boundedness properties of the state transition matrix of LTI system, the inner connection between the finite-frequency controllability Gramian matrix and the state vector is presented as follows:


\begin{definition}[Finite-frequency controllability Gramian, LTI system] \label{con-Gramian-LTI}
Consider the LTI system \eqref{LPV-time} with $(A, B)$ controllable and a
finite-frequency spectrum $\Omega_f$. Then we define: 

\begin{enumerate}[(1)]
  \item  the finite-frequency controllability Gramian as:

\begin{equation}
   \begin{array}{l}
\mathbb{W}(\Omega_f)
=\int_{\Omega_f} 
 \underbrace{(\jmath\omega I-A)^{-1}B e^{\jmath\omega t}
 }_{\mathcal{V}_1(\jmath\omega)}
  \underbrace{e^{-\jmath\omega t} 
 B^*(\jmath\omega I-A)^{-*}
 }_{(\mathcal{V}_1(\jmath\omega))^*}\inva\omega,
   \end{array}
\end{equation}

\item and the state transition matrix-weighted finite-frequency controllability Gramian as: 
\begin{equation}
   \begin{array}{l}
\hat{\mathbb{W}}(\Omega_f)
=\int_{\Omega_f} 
 \underbrace{ e^{At} (\jmath\omega I-A)^{-1}B }_{\mathcal{V}_2(\jmath\omega)}
  \underbrace{
 B^*(\jmath\omega I-A)^{-*} e^{A^*t}
 }_{(\mathcal{V}_2(\jmath\omega))^*}\inva\omega.
   \end{array}
\end{equation}

\end{enumerate}

\end{definition}

Note that these formulas express the well-known controllability Gramian in
system and control considered in a finite-frequency settings.


\begin{proposition}[Finite-frequency controllability
  Gramian and state vector, LTI system case] \label{gramina_LTI}
Consider the finite-frequency controllability Gramian defined in Definition
\ref{con-Gramian-LTI}. Then for a sufficiently long time horizon $t$, the following conclusion holds:  
\begin{enumerate}[(1)]
    \item 
    $\int_0^t \Bigl(
 \iint\limits_{\Omega_f\times \Omega_f} 
\He\Bigl(\mathcal V_1(\jmath\omega)  
(\mathcal V_1(\jmath\tilde\omega))^* \Bigl) 
\inva\omega \inva\tilde{\omega} \Bigr) \inva t$
\\
$= \int_0^t \Bigl( {\mathbb{W}(\Omega_f)} \Bigr)\inva t 
    +\int_0^t \Bigl( \iint\limits_{\Omega_f\times\Omega_f}
    {e^{(\jmath(\omega-\tilde\omega) t} }
    (\jmath\omega I-A)^{-1}  B
    B^* (\jmath\tilde\omega I-A)^{-*} 
   \inva\omega\inva\tilde\omega
    \Bigr)\inva t,$
\\

\item 
$ \int_0^t \Bigl(
 \iint\limits_{\Omega_f\times \Omega_f} 
\He\Bigl(\mathcal V_2(\jmath\omega)  
(\mathcal V_2(\jmath\tilde\omega))^* \Bigl) 
\inva\omega \inva\tilde{\omega} \Bigr) \inva t$
\\
$=   \int_0^t \Bigl( \hat{\mathbb{W}}(\Omega_f) \Bigr)\inva t 
    +\int_0^t \Bigl( \iint\limits_{\Omega_f\times\Omega_f}
     e^{At}
    (\jmath\omega I-A)^{-1}  B
    B^* (\jmath\tilde\omega I-A)^{-*}  {e^{A^*t}}
   \inva\omega\inva\tilde\omega
    \Bigr)\inva t$.

\end{enumerate}

In particular, for AS LPV system, we have 
\begin{enumerate}[(1)]
    \item 
$\lim\limits_{t\to \infty} \frac 1t {\int_0^t  \Bigl(
 \iint\limits_{\Omega_f\times \Omega_f} 
\He\Bigl(\mathcal V_1(\jmath\omega)
(\mathcal V_1(\jmath\tilde\omega))^* \Bigl) 
\inva\omega \inva\tilde{\omega} \Bigr) \inva t}
={\mathbb{W}(\Omega_f)},$

\item
  $\lim\limits_{t\to \infty} \frac{1}{t} {\int_0^t  \Bigl(
 \iint\limits_{\Omega_f\times \Omega_f} 
\He\Bigl(\mathcal V_2(\jmath\omega)
(\mathcal V_2(\jmath\tilde\omega))^* \Bigl) 
\inva\omega \inva\tilde{\omega} \Bigr) \inva t}={0},$

\item
$\lim\limits_{t\to \infty} \frac{1}{t} {\int_0^t  \Bigl(
 \iint\limits_{\Omega_f\times \Omega_f} 
\He\Bigl(\mathcal V_1(\jmath\omega)
(\mathcal V_2(\jmath\tilde\omega))^* \Bigl) 
\inva\omega \inva\tilde{\omega} \Bigr) \inva t}={0},$

\end{enumerate}

where $\mathcal V_1(\jmath\omega)$ and $\mathcal V_2(\jmath\omega)$ are defined
as in \eqref{V-LTI}.

\end{proposition}

\begin{proof}

This conclusion follows directly from the uniform boundedness of
$e^{\jmath\omega t}$ and the \emph{state transition matrix} $e^{At}$.  
For more details, please refer to the proof of \cite[Th. 3]{gKYP_iwasaki2005time}.

\end{proof}

Now, we can explicitly state the one-to-one relation between finite-frequency inputs and an 
IQC condition. This result was used
in the proof of \cite[Th. 3]{gKYP_iwasaki2005time}. Here we state it explicitly
and modified to our purposes as follows:


\begin{proposition}
    [Finite-frequency signals and non-negative $\Omega_f$-IQC, LTI system] 
    \label{prop:IQC-LTI}
Consider the LTI system \eqref{LTI-time} with $A$ Hurwitz and $(A, B)$ controllable. If the input signal $u$ is a finite-frequency signal with respect to $\Omega_f$, 
then the input signal $u\in \mathcal{L}_2$ renders the excited system state $x$ and its derivative $\dot x$ satisfying the following $\Omega_f$-IQC:
 \begin{equation}\label{s_IQC_LTI}
 \int_0^{\infty} 
   \He\Bigl(
   \begin{bmatrix} 
    \dot x^*(t)& x^*(t) 
    \end{bmatrix} 
    \Psi_f 
    \begin{bmatrix} 
    \dot x(t)\\ 
    x(t) 
    \end{bmatrix} 
    \Bigl)
    \inva t \geq 0,
	\end{equation}
where the corresponding weight matrix $\Psi_f$ is defined in \eqref{eq:psif}.

\end{proposition}


\begin{proof}
According to the property of trace, it is easy to see that \eqref{s_IQC_LTI} holds if and only if
\begin{equation}
    \int_0^\infty \tr\bigg(  \He\Bigl(
    \begin{bmatrix} 
    \dot x(t)& x(t) 
    \end{bmatrix} 
    \Psi_f 
    \begin{bmatrix} 
    \dot x^*(t)\\ 
    x^*(t) 
    \end{bmatrix} \Bigl)\bigg) \inva t \geq 0,
\end{equation}
where
\begin{equation}\label{IQC-eqn1}
\begin{array}{l}
\tr\bigg(  \He\Bigl(
\begin{bmatrix} 
    \dot x(t)& x(t) 
    \end{bmatrix} 
    \Psi_f 
    \begin{bmatrix} 
    \dot x^*(t)\\ 
    x^*(t) 
    \end{bmatrix} 
\Bigl)\bigg)
\\
= 
\tr\bigg(  \He\Bigl(
\begin{bmatrix} 
    \dot x(t)& x(t) 
    \end{bmatrix} 
    \begin{bmatrix}
        \Psi_f^{11} & \Psi_f^{12} \\
        \Psi_f^{21} & \Psi_f^{22}
    \end{bmatrix} 
    \begin{bmatrix} 
    \dot x^*(t)\\ 
    x^*(t) 
    \end{bmatrix} 
\Bigl)\bigg)
\\
=
\iint\limits_{\Omega_f\times\Omega_f}  
   \tr\bigg(  \He\Bigl(
   \Psi_f^{11}
   \Bigl(\jmath\omega \mathcal{V}_1(\jmath\omega) + A \mathcal{V}_2(\jmath\omega)\Bigl) \mathcal U(\jmath\omega) \mathcal U^*(\jmath\tilde\omega) 
\Bigl(\jmath\tilde\omega \mathcal{V}_1(\jmath\tilde\omega)
+ A \mathcal{V}_2(\jmath\tilde\omega)\Bigl)^* 
\Bigl)\bigg)
    \inva \omega   \inva \tilde\omega 
\\
+
\iint\limits_{\Omega_f\times\Omega_f}  
   \tr\bigg(  \He\Bigl(
   \Psi_f^{12}
   \Bigl(\jmath\omega \mathcal{V}_1(\jmath\omega)
   + A \mathcal{V}_2(\jmath\omega)\Bigl) \mathcal U(\jmath\omega) \mathcal U^*(\jmath\tilde\omega) 
\Bigl( \mathcal{V}_1(\jmath\tilde\omega)
+ \mathcal{V}_2(\jmath\tilde\omega))\Bigl)^* 
\Bigl)\bigg)
    \inva \omega   \inva \tilde\omega 
\\
+
\iint\limits_{\Omega_f\times\Omega_f}  
   \tr\bigg(  \He\Bigl(
   \Psi_f^{21}
   \Bigl( \mathcal{V}_1(\jmath\omega)
   + \mathcal{V}_2(\jmath\omega)\Bigl) 
   \mathcal U(\jmath\omega)
   \mathcal U^*(\jmath\tilde\omega) 
\Bigl(\jmath\tilde\omega \mathcal{V}_1(\jmath\tilde\omega)
+ A \mathcal{V}_2(\jmath\tilde\omega)\Bigl)^* 
\Bigl)\bigg)
    \inva \omega   \inva \tilde\omega 
\\
+
\iint\limits_{\Omega_f\times\Omega_f}  
   \tr\bigg(  \He\Bigl(
   \Psi_f^{22}
   \Bigl( \mathcal{V}_1(\jmath\omega)
   + \mathcal{V}_2(\jmath\omega)\Bigl) \mathcal U(\jmath\omega)
   \mathcal U^*(\jmath\tilde\omega) 
\Bigl( \mathcal{V}_1(\jmath\tilde\omega)
+ \mathcal{V}_2(\jmath\tilde\omega)\Bigl)^* 
\Bigl)\bigg)
    \inva \omega   \inva \tilde\omega 
    \end{array}
\end{equation}
with $\mathcal V_1$ and $\mathcal V_2$ as defined in \eqref{V-LTI}.

Integrating \eqref{IQC-eqn1} from 0 to $\infty$ and considering Proposition \ref{gramina_LTI}, we obtain 
\begin{equation}\label{IQC-eqn2}
\begin{array}{l}
\int_0^\infty
\tr\bigg(  \He\Bigl(
\begin{bmatrix} 
    \dot x(t)& x(t) 
    \end{bmatrix} 
    \Psi_f 
    \begin{bmatrix} 
    \dot x^*(t)\\ 
    x^*(t) 
    \end{bmatrix} 
\Bigl)\bigg) 
\inva t
\\
=\int_0^\infty 
\tr \Bigl(
\int\limits_{\Omega_f}  
   \underbrace{\begin{bmatrix} 
    \jmath\omega  &     1    
    \end{bmatrix}  \Psi_{f}
    \begin{bmatrix}    
    \jmath\omega &     1
    \end{bmatrix}^*}_{\geq 0}
      \mathcal V_1(\jmath\omega)
      \mathcal U(\jmath\omega)
   \mathcal U^*(\jmath\omega) 
    (\mathcal V_1 (\jmath\omega))^*
    \inva \omega
     \Bigl)
     \inva t
    \\
   +
    \int_0^\infty 
     \tr \Bigl(
   \iint\limits_{\Omega_f\times\Omega_f, \omega\neq \tilde\omega}  
   \He\Bigl(
   \begin{bmatrix} 
    \jmath\omega  &     1    
    \end{bmatrix}  \Psi_{f}
    \begin{bmatrix}    
    \jmath\tilde\omega &     1
    \end{bmatrix}^*
   \mathcal V_1(\jmath\omega) 
     \mathcal U(\jmath\omega)
   \mathcal U^*(\jmath\tilde\omega) 
    (\mathcal V_1 (\jmath\tilde\omega))^*
    \Bigl)
    \inva \omega \inva \tilde \omega
    \Bigl)
    \inva t
    \\
+
\int_0^\infty 
\tr\bigg(
 \iint\limits_{\Omega_f\times\Omega_f}
  \He \Bigl(  
\begin{bmatrix}
    A^*(\mathrm p(t)) & I
\end{bmatrix}

\begin{bmatrix}
   A(\mathrm p(t))\\
   I
\end{bmatrix}
    \mathcal V_2 (\jmath\omega) 
     \mathcal U(\jmath\omega)
   \mathcal U^*(\jmath\tilde\omega) 
    (\mathcal V_2 (\jmath\tilde\omega))^* 
    \Bigl)
    \inva \omega \inva \tilde \omega
    \bigg)
 \inva t
    \\
+ 
\int_0^\infty 
\tr\bigg(
\iint\limits_{\Omega_f\times\Omega_f}
\He \Bigl(
\begin{bmatrix}
    A^*(\mathrm p(t)) & I
\end{bmatrix}
(\Psi_f\otimes I)
\begin{bmatrix}
    \jmath\omega I\\
    I
\end{bmatrix}
    \mathcal V_1 (\jmath\omega) 
     \mathcal U(\jmath\omega)
   \mathcal U^*(\jmath\tilde\omega) 
   (\mathcal V_2 (\jmath\tilde\omega))^* 
    \Bigl)
    \inva \omega \inva \tilde \omega
   \bigg)
 \inva t
    \\
+ 
\int_0^\infty 
\tr\bigg(
\iint\limits_{\Omega_f\times\Omega_f}
\He\Bigl(
\begin{bmatrix}
    -\jmath\tilde\omega I &     I
\end{bmatrix}
(\Psi_f\otimes I)
\begin{bmatrix}
    A(\mathrm p(t)) \\
    I
\end{bmatrix} 
    \mathcal V_2 (\jmath\omega) 
      \mathcal U(\jmath\omega)
   \mathcal U^*(\jmath\tilde\omega) 
   (\mathcal V_1 (\jmath\tilde\omega))^* 
   \Bigl)  
   \inva \omega \inva \tilde \omega
   \bigg)
 \inva t,
    \end{array}
\end{equation}

Combining \eqref{IQC-eqn2} and Proposition \ref{gramina_LTI}, the last three terms are uniformly bounded, then non-negativity of \eqref{s_IQC_LTI} can be verified.

\end{proof}

Proposition \ref{prop:IQC-LTI} is a key result for establishing the gKYP in the
time-domain version of \cite{gKYP_iwasaki2005time}. We highlight these relations
in a diagram below but first recall the second version of the gKYP.


\begin{lemma}[GKYP lemma, Iwasaki and Hara, {\it System $\&$ Control Letter }\cite{gKYP_iwasaki2005generalized}]
\label{lem:gKYP-TD}
Consider the LTI system \eqref{LTI-time} with $A$ Hurwitz and $(A, B)$ controllable, and a finite-frequency range $\Omega_f$ 
with the corresponding weight matrix $\Psi_f$ as defined in \eqref{eq:psif}. Assume $\Pi \in  \mathbb R^{(q+p)\times (q+p)}$ is Hermitian, 
and the finite-frequency input signal $u\in \mathcal{L}_2$
renders the corresponding system's state and its derivative satisfying the $\Omega_f$-IQC in \eqref{s_IQC_LTI} (as in Proposition \ref{prop:IQC-LTI}), 
then the following conclusions are equivalent:

\begin{enumerate}[(1)]
  \item
There exist Hermitian matrices $P$, $Q\geq0$ such that
 \begin{equation} 
\begin{bmatrix}
				A &  B\\
				I& 0 
 \end{bmatrix}^*
              (\Theta \otimes P +\Psi_f \otimes Q)    
              \begin{bmatrix}
				A &  B \\
				I & 0
              \end{bmatrix}
      +
      \begin{bmatrix}
                C & D  \\
				0 & I
			\end{bmatrix}^* \Pi 
      \begin{bmatrix}
				C    & D   \\
				0 & I 
			\end{bmatrix}   \leq 0,~{\rm with}~\Theta=\begin{bmatrix}
    0& 1\\
    1 & 0
\end{bmatrix}.
	\end{equation}

  \item
The input-output performance in time domain \eqref{IO-Performance-TD} holds.

\end{enumerate}

\end{lemma}

The relationship between the input signals and the input-output performance
involved in KYP lemma \ref{lem:KYP}, and the relationship between the input
signals, time domain interpretation as well as the input-output performance
involved in the frequency domain gKYP lemma \ref{lem:gKYP-FD} and the time
domain gKYP Lemma \ref{lem:gKYP-TD} is schematically illustrated in Fig. \ref{fig:IO-LTI-diagram}.

\begin{figure}[htp]
\centering
    \includegraphics[width=1\linewidth]{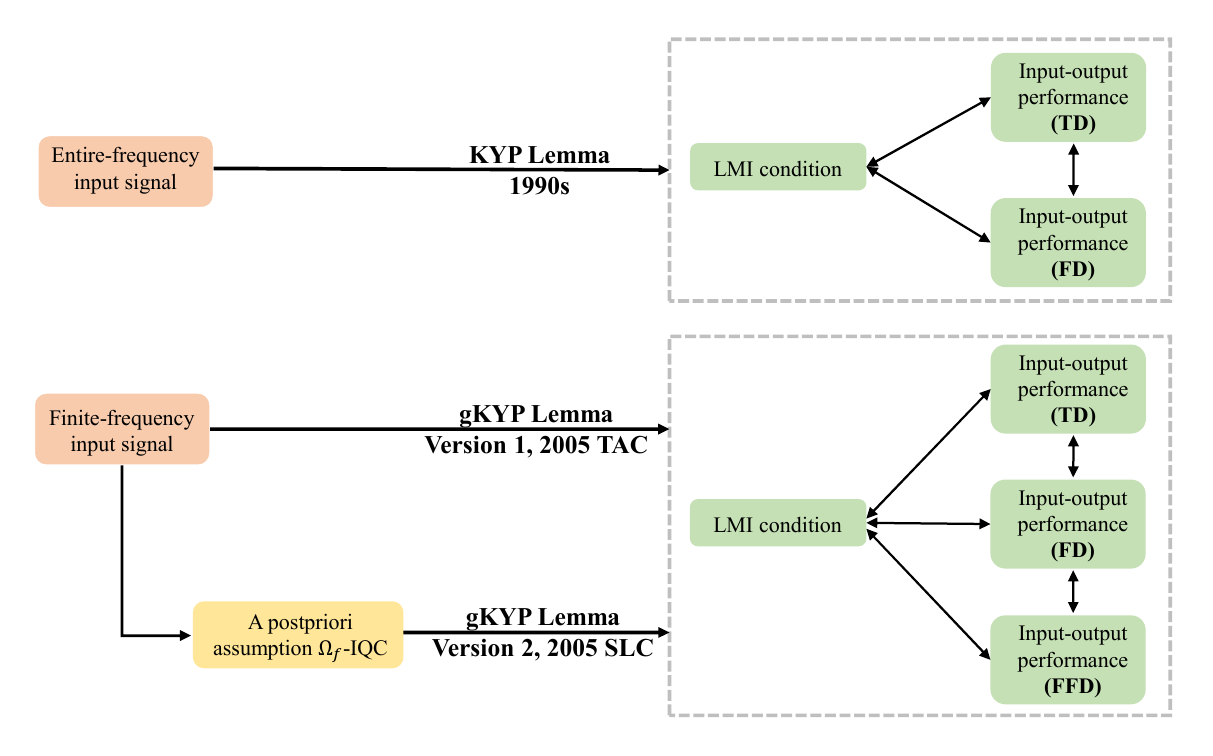}
    \caption{Overview on the two gKYP approaches to finite-frequency analysis of
    general LTI systems in finite-frequency domain (FFD), (entire) frequency domain (FD) and time domain (TD) as well as
  the intermediate \emph{a posteriori} condition of Proposition
\ref{prop:IQC-LTI}. For comparison, we include the entire frequency domain KYP.}
    \label{fig:IO-LTI-diagram}
    \end{figure}

\section{LPV systems, existing attempts, wrong points}
\label{sec:LPV}

\subsection{LPV systems, notions, preliminaries, stability}

We now consider LPV systems of the form
\begin{equation}\label{LPV-time}
	\begin{split}
		\dot{x}(t) &= A (\mathrm p(t))x(t)+B (\mathrm p(t)) u(t),\\
		y(t) &= C (\mathrm p(t))x(t)+D (\mathrm p(t)) u(t),
	\end{split}
\end{equation}
where $\mathrm p$ is a continuously differentiable time-varying parameter curve in $\mathbb R^l$.
We assume that the parameter $\mathrm p:=(\mathrm p_1~ \mathrm p_2 \dotsc \mathrm p_l)\in
\mathbb {R}^l$ and its rate $\dot{\mathrm p}(t):=(\dot{\mathrm p}_1~\dot{\mathrm p}_2 \dotsc
\dot{\mathrm p}_l)\in \mathbb {R}^l$ are bounded with known upper bounds and
lower bounds on the parameter $\mathrm p$ and its rate of change $\dot{\mathrm p}$.

\begin{definition}[Uniform spectral radius]
The uniform spectral radius of a time-varying matrix 
$A(\mathrm p(t))\in\mathbb{C}^{n\times n}$ over a parameter set $\mathcal{P}$ is defined as:
\[
  \rho_{\rm unif}(A(\mathrm p(t))) = \sup\limits_{\rho\in\mathcal{P}} \rho(A(\mathrm p(t)))
  = \sup\limits_{t \geq 0}\max\limits_{i}|\lambda_i(A(\mathrm p(t)))|.
\]

\end{definition}

In many works, some of the assumptions mentioned for the plant are directly related to the \emph{bounded-input-bounded-state} (BIBS)/\emph{bounded-input-bounded-output} (BIBO) stable \cite{UAS_Rugh1996linear}, \emph{input-to-state stability} (ISS) \cite{IS_Sahan2024UAS}, \emph{input-to-output stability} (IOS) \cite{IO_Karafyllis2021lyapunov}, \emph{{uniformly asymptotically stable (UAS)}}\cite{UAS_zhou2015on}, and also system \emph{controllability/observability} \cite{UAS_Rugh1996linear}, which
are at the basis of important stability properties for systems with input
signals.
However, unlike LTI systems, where the notions of stability are equivalent, see,
e.g., \cite{UAS_Rugh1996linear}, 
e.g. asymptotically stable LPV systems may not be uniformly asymptotically
stable (UAS); see, e.g., the
examples in \cite{UAS_importance_1, UAS_importance_2}.
We recall the basic definitions that apply to LPV systems in particular:

\begin{definition}
[\cite{UAS_Rugh1996linear}, Def. 12.1] 
\label{def:BIBO}
The LPV system \eqref{LPV-time} is said to be  
   \emph{uniformly bounded-input, bounded-state (BIBS)} stable if for all
   $t_0\geq 0$, $\delta >0$, there exists a positive constant $\epsilon$, such
   that
   \[\|u(t)\|\leq \delta \Rightarrow  \|x(t)\|\leq \epsilon, ~\forall t\geq t_0.\]

In other words, there exist a finite constant $\eta$ such that the input and output responses satisfies
\[\sup_{t\geq t_0} \|y(t)\| \leq \eta \sup_{t\geq t_0} \|u(t)\|.\]

\end{definition}


\begin{definition}
[
\cite{UAS_zhou2015on}, Def. 1]
\label{def:UAS}
    The autonomous LPV system \eqref{LPV-time} 
    is said to be 
\begin{enumerate}[(1)]
  \item   
\emph{{stable}} if for all $t_0 \geq 0$, $\epsilon >0$, there exists
$\delta(\epsilon,t_0)> 0$, such that
   \[\| x(t_0) \|\triangleq \| x_0 \| < \delta(\epsilon,t_0) \Rightarrow \| x(t) \| < \epsilon,~\forall t\geq t_0;\]
   
\item 
\emph{{uniformly stable  (US)}} if $\delta$ in item (1) is independent of $t_0$;

\item
\emph{{uniformly asymptotically stable (UAS)}} if it is uniformly stable and $\zeta(t_0)$ and $T(t_0,\epsilon)$ in item (3) is independent of $t_0$.
\end{enumerate}
   \end{definition}

We have to confirm that 
the state vector approaches the origin at the speed of the same
level in order to prove that the equilibrium is UAS. Obviously, UAS is much
stricter than the definition of (merely) stable and asymptotically stable. From
a practical perspective, the property of UAS rules out the possibility of having
an arbitrarily slow convergence of solutions to the origin or an arbitrarily
large transient overshoot when initial states are confined to a bounded set. 



\subsection{Existing attempts, inconsistency, 
counterexample}

The insufficiency of the LTI vertex systems-based attempts is due to the inter-modulation between the scheduling time-varying parameter $\mathrm{p}$ and the system state/input/output signals, which is non-negligible in finite-frequency analysis. 
This effect becomes obvious in the following example with $\mathrm{p}(t)=\cos(w_0 t)$ and
exemplary input $u(t) = \sin(w_0t)$, so that in this case the inter-modulation in 
\[\begin{array}{l}
  \dot x(t) = Ax(t) + \mathrm{p}(t)Bu(t) = Ax(t) + B\cos(w_0t)\sin(w_0t) = Ax(t) +
  \frac 12 B\sin(2w_0t),
\end{array}\]
produces a signal of double frequency so that an \emph{out-of-band} response is
straight forward to construct for any limited frequency case. Accordingly, LPV systems come with particular needs unlike LTI systems in which finite
frequency input signals will only produce finite-frequency output signals in the
same domain. 
In fact, many researchers have pointed out that the input-output relation in
frequency domain of LPV system cannot be represented by a simple
transfer function \cite{schoukens2019frequency}. Under mild assumptions on the
scheduling parameter, several relevant concepts including frozen frequency
response function, instantaneous frequency response function and harmonic
frequency response function on the basis of the impulse response or the time-dependent impulse response had been reported
\cite{FRF_toth2010modeling, FRF_Louarroudi2014frequency}. To comprehensively
characterize the frequency response of general LPV systems, one may resort to
the recently established Volterra-series representation based sequence
of generalized transfer function(s); see, e.g., \cite{gosea2021reduced}
and the references therein.

Also in the time domain, the gKYP
lemma in the form of \cite{gKYP_iwasaki2005time} provides a referable and
utilizable entrypoint for finite-frequency analysis of LPV systems. 
The first attempt on this path was presented in \cite{S_ding2010fuzzy} for
discrete-time fuzzy systems, which is a special form of LPV systems with polytopic
parameter dependencies. There, the authors introduced
a discrete-time matrix-valued IQC to represent the
admissible input signals, which achieves positivity of the discrete-time matrix-valued IQC.
Following a similar routine, the works \cite{S_sun2013induced, S_Wang2017H,S_baar2020parameter,S_li2014fault} derived the counterparts of \cite{S_ding2010fuzzy} for
continuous-time LPV systems, wherein sufficient LMI conditions are given for
input-output performance analysis. For comparison, the main results of
\cite{S_sun2013induced} and \cite{S_baar2020parameter} are consolidated with our
notation and included as the following lemma.

\begin{lemma}\label{lemma_LPV_FF}
Assume the LPV system \eqref{LPV-time} is AS and consider inputs with on a
finite frequency range $\Omega_f$. 
Let the following $\Omega_f$-IQC hold:
\begin{equation}
\label{eq:L2_LPV_matrix-ineq}
    \int_0^{\infty}  
    \begin{bmatrix}
     \dot x(t) & x(t)\\ 
    \end{bmatrix} \Psi_f \begin{bmatrix}
        \dot x^*(t) 
        \\ x^*(t)      
    \end{bmatrix} \inva t \geq 0.
\end{equation}
Then, if there exists a Hermitian matrix $Q>0$, $P(\mathrm p(t))$ and a differentiable map $P\colon \mathbb R^{r}\to \mathbb R^{n\times
  n}$ so that
 \begin{equation}\label{L2_FF_LPV}
		\begin{split}
               \begin{bmatrix}
				A(\mathrm p(t)) &  B(\mathrm p(t))\\
				I& 0 
			 \end{bmatrix}^*(\Theta \otimes P(\mathrm p(t))  + \Theta_d \otimes \dot P(\mathrm p(t))+\Psi_f \otimes Q)
              \begin{bmatrix}
				A(\mathrm p(t)) &  B(\mathrm p(t))\\
				I & 0
			\end{bmatrix}  &
            \\ 
            +\begin{bmatrix}
				C (\mathrm p(t))   & D (\mathrm p(t))  \\
				0 & I 
			\end{bmatrix}^*\Pi \begin{bmatrix}
				C (\mathrm p(t))   & D (\mathrm p(t))  \\
				0 & I 
			\end{bmatrix}   &\leq 0,
		\end{split}
	\end{equation}
for some performance index matrix $\Pi \in  \mathbb R^{(q+p)\times (q+p)}$, $\Psi_f$ as defined in \eqref{eq:psif}, $\Theta=\begin{bmatrix}
    0& 1\\
    1 & 0
\end{bmatrix}$ and $\Theta_d=\begin{bmatrix}
    0& 0\\
    0 & 1
\end{bmatrix}$, then the given LPV system satisfies input-output performance inequalities \eqref{IO-Performance-TD} and \eqref{IO-Performance-FD}.

\end{lemma}

If one drops the finite-frequency restriction, the input-output performance analysis of LPV systems can readily be based on the following lemma.

  \begin{lemma}[\cite{con_gahinet1996affine} or Eq. (14) with $Q=0$ in \cite{S_sun2013induced}] \label{lemma_LPV_EF}
    Assume the LPV system
    \eqref{LPV-time} is AS. If there exists a differentiable function $P\colon \mathbb R^{r}\to \mathbb R^{n\times
  n}$, so that $P(\mathrm p(t))$ is Hermitian and positive definite, and  
 \begin{equation}\label{L2_LPV_EF}
		\begin{split}
               \begin{bmatrix}
				A(\mathrm p(t)) & B(\mathrm p(t))\\
				I& 0 
			 \end{bmatrix}^*
             (\Theta \otimes P(\mathrm p(t)) + \Theta_d \otimes \dot P(\mathrm p(t)))
              \begin{bmatrix}
			A(\mathrm p(t)) & B(\mathrm p(t))\\
				I &0
			\end{bmatrix}  &
            \\ 
            +\begin{bmatrix}
				C (\mathrm p(t))   & D (\mathrm p(t))  \\
				0 & I 
			\end{bmatrix}^*\Pi \begin{bmatrix}
				C (\mathrm p(t))   & D (\mathrm p(t))  \\
				0 & I 
			\end{bmatrix}   &\leq 0,
		\end{split}
	\end{equation}
 for some performance index matrix $\Pi \in  \mathbb R^{(q+p)\times (q+p)}$, $\Theta$ and $\Theta_d$ as given in Lemma \ref{lemma_LPV_FF}, then the given LPV system satisfies input-output performance inequalities \eqref{IO-Performance-TD} and \eqref{IO-Performance-FD}.
\end{lemma}


The LMI condition \eqref{L2_FF_LPV} have a similar form as the LMI condition \eqref{gKYP_LMI} of the gKYP lemma. 
As emphasised earlier, however, the IQC non-negativity property 
is not practical, as 
it can only be tested \emph{a posteriori}, i.e., only when the corresponding state to the input signals are known.
Therefore, for LTI systems, we had shown (cp. Proposition \ref{prop:IQC-LTI}) that finite-frequency input signal always renders the $\Omega_f$-IQC be non-negative. 
And this additional condition has become a practical tool for finite-frequency analysis.
Unfortunately, the same \emph{free travel pass} is not available for LPV system as illustrated in the following example.

\begin{example}[Counterexample]\label{exa:counterexample}
Consider the LPV system \eqref{LPV-time} in affine form with the following parameter matrices:   
\begin{equation}\label{eq:counterexample-LPV-2}
    \begin{split}
\{{A_0,A_1}\}&=\left\{\begin{bmatrix}
		 -8.6329  & -6.5229 \\
   -1.2735  & -9.4779
	\end{bmatrix}, \begin{bmatrix}
		-2.5827  &  7.1275 \\
    7.8186  & -1.9513
	\end{bmatrix} \right\},\\
\{{B_0,B_1}\} &=\left\{\begin{bmatrix}
	-19.6836 \\
   16.7629
	\end{bmatrix},\begin{bmatrix}
	-3.7921 \\
    8.0760
	\end{bmatrix} \right\},\\
\{{C_0,C_1}\}&=\left\{\begin{bmatrix}
		-1.5715  &  1.5934
	\end{bmatrix}, \begin{bmatrix}
		4.2725  & -4.3798
	\end{bmatrix} \right\},\\
\{{D_0,D_1}\}& =\left\{ -4.6104, 1.8747 \right\},
     \end{split}
\end{equation}
where the upper/lower bound of the time-varying scheduling parameter $\mathrm{p}$ is set as $[\underline{\mathrm{p}}, \overline{\mathrm{p}}]$ with $\underline{\mathrm{p}}=0.1, \overline{\mathrm{p}}=0.2$, the upper/lower bound of its variation rate is set as $[\underline{\dot{\mathrm{p}}}, \overline{\dot{\mathrm{p}}}]$ with $\underline{\dot{\mathrm{p}}}=0.4,\overline{\dot{\mathrm{p}}}=0.6$, and with input signals in the
	low-frequency range interval $\Omega_f:=[-1,1]$. 
Again, we use the $H_\infty$ performance index matrix $\Pi=\diag\{I,-\gamma^2 I\}$. 
Following the numerical solving approach presented in the next section, we
compute the optimal $\mathcal{L}_2$-norm induced gain so that
\[
\begin{array}{l}
\sqrt{\frac{\int_{0}^{\infty} y^*(\tau) y(\tau) \inva \tau}  {\int_{0}^{\infty}
u^*(\tau) u(\tau) \inva \tau} }<\gamma_f^{\star}\approx 3.7767.
\end{array}
\]
Then for the considered LPV system with the low-frequency inputs so that the parameter-dependent LMI 
\eqref{L2_FF_LPV} constraint feasibility problem is fulfilled with 
$\Pi = \diag\{I,-(\gamma_f^{\star})^2 I\}= 
\diag\{I, -(3.7767)^2 I \}.$

Also, the minimal performance upper bound $\gamma_e^{\star}$ for entire-frequency interval $\Omega_e:=(-\infty,\infty)$ can be obtained by solving the minimization LMI \eqref{L2_FF_LPV} with $Q=0$, and let $\gamma_e$ be the decision variable to be minimized, i.e.,
\[\begin{array}{l}
\sqrt{\frac{\int_{0}^{\infty} y^*(\tau) y(\tau) \inva \tau}  {\int_{0}^{\infty} u^*(\tau) u(\tau) \inva \tau} }<\gamma_e^{\star} \approx 5.2445.
\end{array}\]

Now, for the inbound frequency signal 
\[\begin{array}{l}
u(t)=\cos(t+8)+\cos(t+10)+\cos(t+20),
\end{array}\]
we can test the validity of the performance inequality (see Fig.
\ref{fig:E1-gamma}), but also observe the failure of the feasibility constraint
\eqref{eq:L2_LPV_matrix-ineq}; see Fig. \ref{fig:E1-S}.

\begin{figure}[htbp]
		\centering
		\includegraphics[width=0.7\textwidth]{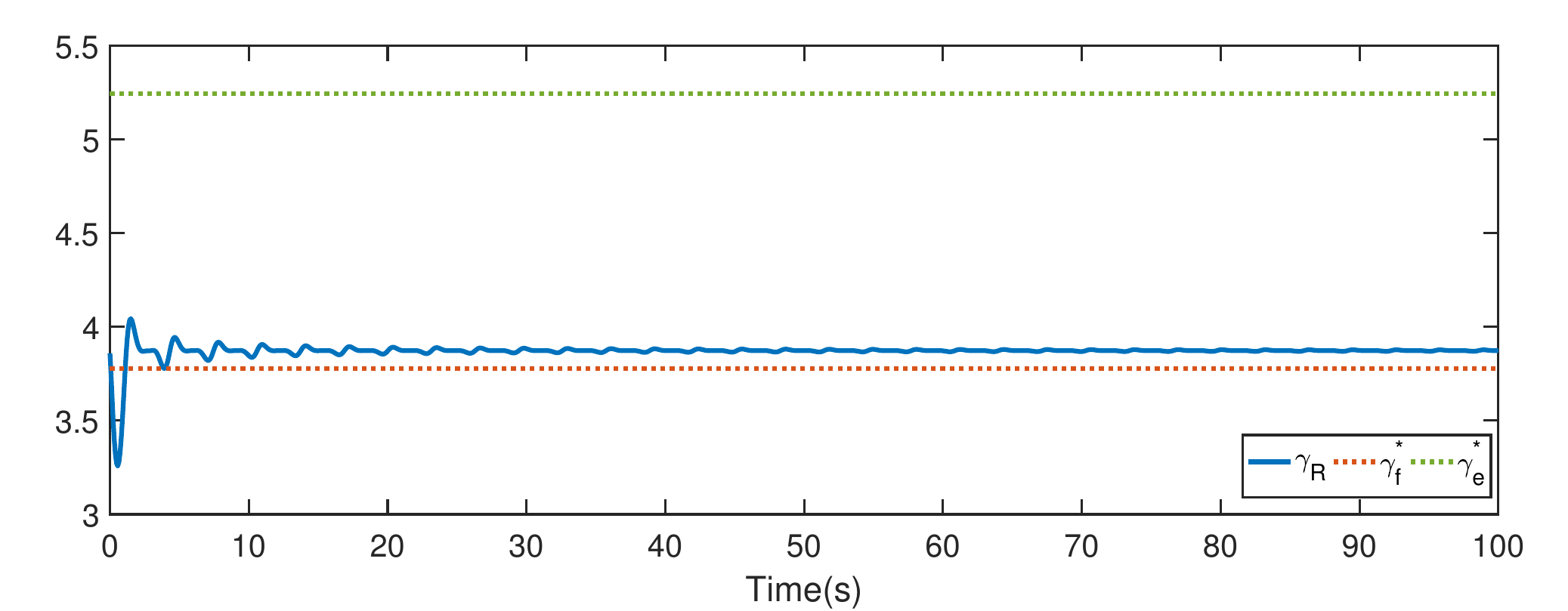}
		\caption{Actual input-output dynamical behavior 
  ($\gamma_R$) 
  and the optimal performance bounds $\gamma_f^{\star}$ and $\gamma_e^{\star}$.}\label{fig:E1-gamma}
	\end{figure}

\begin{figure}[htbp]	
		\centering
		\includegraphics[width=0.7\textwidth]{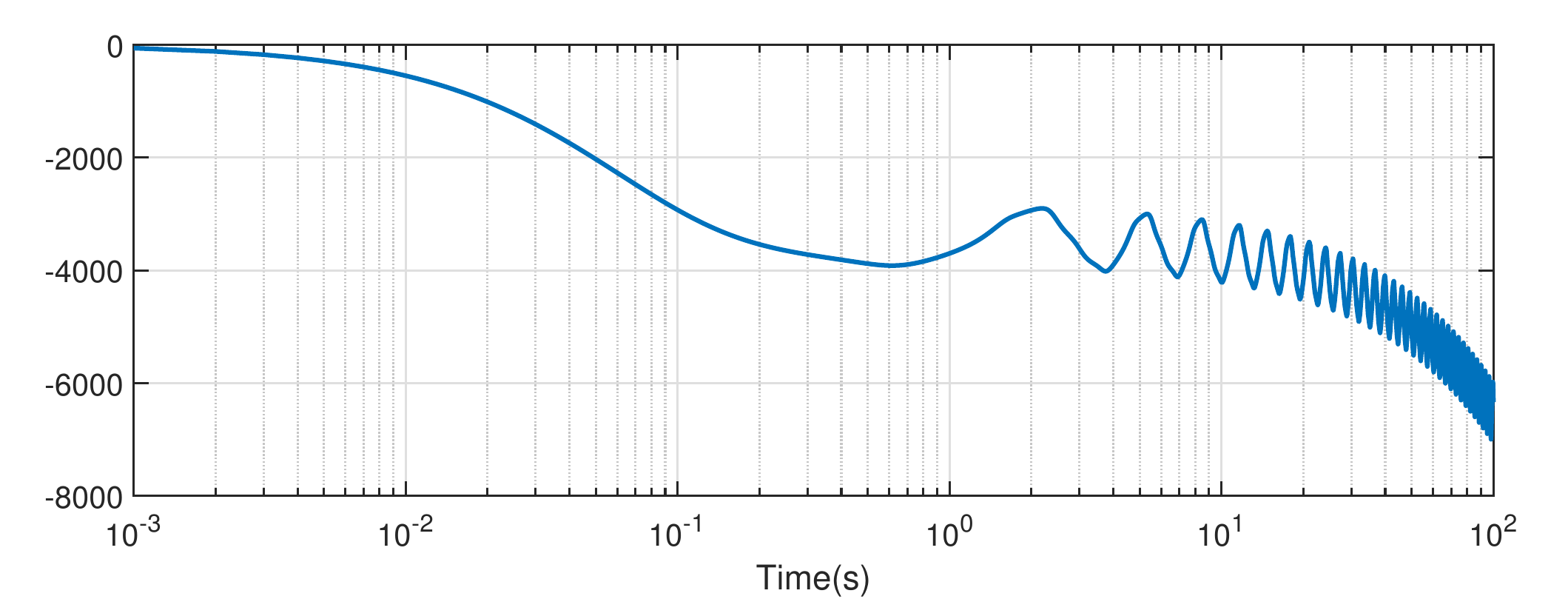}
		\caption{The scalar valued $\Omega_f$-IQC with finite-frequency range.}
		\label{fig:E1-S}
	\end{figure}

 \end{example}

\begin{remark}
  We note that in our setup the signals are truncated for $t<0$ so that, despite
  their appearance as a sinusoid with a single frequency, they do not strictly
  obey the frequency limited assumption. 
  Nonetheless, it is common practice (because of infinite horizons can not be
  simulated and because of the decay of the frequency components in the
  one-sided Fourier transform) to also consider finite-frequency signals on
  finite time ranges.
  Same holds for the $\mathcal{L}_2$ integrability that is a theoretical issue since,
  e.g., the energy of a sine wave is not finite over an unbounded time domain.
  Also here, it is practical to consider the signals as $\mathcal{L}_2$ integrable
  because of the finite simulation domain. 
  Furthermore, on a time interval like $[0,\infty)$, it is always possible to
  introduce a \emph{discount} factor $e^{-\lambda t}$ with $0 < \lambda\ll 1$
  that will render the signal square integrable with only a small effect on its
  frequency spectrum.
\end{remark}

This counterexample shows that there exists intrinsic difference between LPV
systems and LTI systems, with respect to the relationship between
finite-frequency input signals and the non-negative IQC rendering signals.
This difference makes the practical applicability of Lemma \ref{lemma_LPV_FF}
for solving finite-frequency analysis and synthesis problems of LPV systems be
lost, at least, be unreliable. 
With the goal of deriving a similarly valid result for finite-frequency analysis, 
an enlarged frequency interval, based on the interaction/location gap between the system poles and the frequency range, as well as a group of controllability Gramians, is introduced for LPV systems.

\section{Boundedness theorem, extension of gKYP lemma}

\subsection{Boundedness of LPV systems, from general stable LPV systems to BIBS/UAS LPV systems}

Consider LPV systems \eqref{LPV-time} with an initial state $x_0=x(t_0)$, the unique solution of the system has the representation
\begin{equation}\label{x-LPV}
\begin{array}{l}
x(t)=\Phi(t,t_0) x(t_0)+\int_{t_0}^{t}\Phi(t,\tau)B(\mathrm p(\tau))u(\tau)
\inva\tau,~ \forall t\geq t_0,
\end{array}
\end{equation}
where $\Phi(t,t_0)$ is the state transition matrix of the LPV system and satisfies the following differential equations
\begin{equation}\label{Phi_dot}
\begin{array}{l}
    \dot{\Phi}(t,t_0)=A(\mathrm p(t))\Phi(t,t_0),\\
    \dot{\Phi}(t_0,t)=-\Phi(t_0,t)A(\mathrm p(t)),~\Phi(t_0,t_0)=I.
\end{array}
\end{equation}

The boundedness of state transition matrix is given in the following Lemmas.

\begin{lemma}[Boundedness properties of BIBS stable LPV system, \cite{BIBS_2017bounded}]\label{boundness:BIBS}
Consider an $\mathcal{L}_2$ bounded Hermitian matrix $M(\tau)$, and the state transition matrix $\Phi(t,t_0)$ of LPV system. 
If the LPV system is (uniformly) BIBS stable, then there exist positive scalars $n_1$, $n_2$ and $n_3$ such that the following hold: 
		\begin{enumerate}[(1)]
			\item  $
				\|\Phi(t,t_0)\|< n_1, ~t\geq t_0\geq 0$;
                
			\item			
            $\int_0^t \|\Phi(t,\tau) B(\tau)\| \inva\tau < n_2,~t\geq\tau\geq 0$;

            \item	
            $\int_0^t \|\Phi(t,\tau) M(\tau)\Phi^*(t,\tau)\| \inva\tau < n_3,~t\geq\tau\geq 0$.
		\end{enumerate}
	
\end{lemma}

Generally speaking, the upper bound $n_1$, $n_2$ and $n_3$ are not readily
available just as state transition matrix $\Phi(t,\tau)$. However, there exist
algorithms for approximxating the bounds; refer to, e.g., \cite{Adrianova1995}. 

\begin{lemma}[Boundedness properties of UAS LPV system] \label{boundness:UAS}
Consider an $\mathcal{L}_2$ bounded Hermitian matrix $M(\tau)$, and the state transition matrix $\Phi(t,t_0)$ of LPV system. If LPV system \eqref{LPV-time} is \emph{UAS}. 
Then there exist positive scalars $\alpha$, $\beta$ and $m$, such that the following are uniformly bounded: 
\begin{enumerate}[(1)]
  \item
  $\|\Phi(t,t_0)\| 
    \leq \alpha e^{-\beta (t-t_0)},~t\geq t_0\geq 0$;
\item 
$\|\int_0^t \Phi(t,\tau) \inva\tau\|
\leq \frac{\alpha}{\beta},~t\geq \tau\geq 0$;
\item 
$\|\int_0^t \Phi(t,\tau) M(\tau) \Phi^*(t,\tau) \inva\tau \|
\leq \frac{\alpha^2}{2\beta} m$, with $\|M(\tau)\|\leq m$.
\end{enumerate}

\end{lemma}

\begin{proof}
With the estimate of [\cite{lemma_agulhari2018A}, Lem. 1], relation $\|\Phi(t,t_0)\| \leq \alpha e^{-\beta (t-t_0)}$ follows directly. 
The other inequalities follow by the same arguments and that the norm of an integral is less than the integral of the norm as it follows, e.g., from Minkowski's inequality; see \cite[Prop. 1.3]{Min_bahouri2011fourier}.

\end{proof}

\begin{remark}
    
In particular, the scalars $\alpha$, $\beta$ can be obtained by using any feasible solution of the following parameter-dependent Lypapunov matrix inequalities
[\cite{UAS_Rugh1996linear}, Th. 7.2]. i.e.,  
\begin{equation}\label{UAS-1}
		c_1 I\leq P_s(\mathrm p(t))\leq c_2 I,~c_1,c_2\in\mathbb{R}^+,
	\end{equation}
	\begin{equation}\label{UAS-2}
		A^T(\mathrm p(t)) P_s(\mathrm p(t)) +P_s(\mathrm p(t)) A(\mathrm p(t)) +\dot{P}_s(\mathrm p(t)) \leq -c_3 I, ~c_3\in\mathbb{R}^+,
	\end{equation}
which is explicitly characterized by the parameter mappings $\alpha=\frac{c_2}{c_1}$ and $\beta=\frac{c_3}{2c_2}$.

\end{remark}

Then for the initial state is zero and a finite-frequency input signal, i.e., 
\begin{equation}\label{u}
\begin{array}{l}
u(t)=\int_{\Omega_f} \mathcal U(\jmath\omega) e^{\jmath \omega t} \inva{\omega},~\omega\in\Omega_f,
\end{array}
\end{equation}
the state vector \eqref{x-LPV} and its derivative can be described as 
\begin{equation}\label{x-Dx}
    \begin{array}{l}
        x(t)       
        =\int_{\Omega_f}
        (\mathcal V_{\mathrm p}^1(\jmath\omega) +\mathcal V_{\mathrm p}^2(\jmath\omega)
        +\mathcal V_{\dot {\mathrm p}}(\jmath\omega))
        \mathcal U(\jmath\omega) \inva\omega,
        \\
        \dot x(t)  =\int_{\Omega_f}
        (\jmath\omega\mathcal V_{\mathrm p}^1(\jmath\omega)
        + A(\mathrm p(t) )\mathcal V_{\mathrm p}^2(\jmath\omega) +A(\mathrm p(t))\mathcal V_{\dot {\mathrm p}}(\jmath\omega)) \mathcal U(\jmath\omega) \inva\omega,
       \end{array}
\end{equation}
where $\mathcal V_{\mathrm p}^1 (\jmath\omega)$, $\mathcal V_{\mathrm p}^2 (\jmath\omega)$ and $\mathcal V_{\dot {\mathrm p}}(\jmath\omega)=\mathcal V_{\dot {\mathrm p}}^1(\jmath\omega)+\mathcal V_{\dot {\mathrm p}}^2(\jmath\omega)$ are defined as
   \begin{equation}\label{V-LPV}
   \begin{array}{l}
\mathcal V_{\mathrm p}^1(\jmath\omega)
=e^{\jmath\omega t} (\jmath\omega I-A(\mathrm p(t)))^{-1} B(\mathrm p(t)),
\\
\mathcal V_{\mathrm p}^2(\jmath\omega)
= -\Phi(t,0) (\jmath\omega I-A(\mathrm p(t)))^{-1} B(\mathrm{p}(0)),
   \\
\mathcal V_{\dot {\mathrm p}}^1(\jmath\omega) = 
-\int_{0}^{t}
\Phi(t,\tau) (A(\mathrm p(t))- A(\mathrm p(\tau)))  e^{\jmath \omega \tau} (\jmath \omega I -A(\mathrm p(t)))^{-1} B(\mathrm p(\tau))   \inva\tau,
\\
\mathcal V_{\dot {\mathrm p}}^2(\jmath\omega)= 
-\int_{0}^{t}
{\Phi(t,\tau)}  e^{\jmath \omega \tau} (\jmath \omega I -A(\mathrm p(t)))^{-1} \dot{B}(\mathrm p(\tau)) \inva\tau.
   \end{array}
\end{equation}

\begin{definition}[Finite-frequency controllability Gramian, LPV system] 
\label{con-Gramian-LPV}
Consider the LPV system \eqref{LPV-time} with $(A(\mathrm p(t)), B(\mathrm p(t)))$ controllable, and finite-frequency spectrum $\Omega_f$. Then the following terms are referred to as finite-frequency controllability Gramian, which can be viewed as a generalization of the well-known controllability Gramian in system and control, from entire-frequency setting to finite-frequency setting: 

\begin{enumerate}
    \item finite-frequency controllability Gramian as: 
  \begin{equation}
  \begin{array}{l}
{\mathbb{W}_{\mathrm{p}}(\Omega_f)}
=\int\limits_{\Omega_f} 
 \underbrace{e^{\jmath\omega t}(\jmath\omega I-A(\mathrm p(t)))^{-1}
  B(\mathrm p(t))}_{\mathcal V_p^1(\jmath\omega)}  \underbrace{B^*(\mathrm p(t)) 
  (\jmath\omega I-A(\mathrm p(t)))^{-*} e^{-\jmath\omega t}}_{(\mathcal V_p^1(\jmath\omega))^*} 
   \inva\omega;
\end{array}
\end{equation}

\item  state-transition-matrix-weighted finite-frequency controllability Gramian as:

  \begin{equation}\label{W_p_dp}
  \begin{array}{l}
{\hat{\mathbb{W}}_{\mathrm{p}}(\Omega_f)}
= \int\limits_{\Omega_f} \underbrace{\Phi(t,0)(\jmath\omega I-A(\mathrm p(t)))^{-1}
  B(\mathrm p(0))}_{\mathcal V_p^2(\jmath\omega)} 
 \underbrace{B^*(\mathrm p(0))(\jmath\omega I-A(\mathrm p(t))^*\Phi^*(t,0)}_{(\mathcal V_p^2(\jmath\omega))^*} 
\inva\omega;
\end{array}
\end{equation}

\item  shifted finite-frequency controllability Gramian as: 
\begin{equation}
\begin{array}{l}
{\mathbb{W}_{\dot{\mathrm{p}}}(\Omega_f)}
={\mathbb{W}_{\dot{\mathrm{p}}}^1(\Omega_f)}
+{\mathbb{W}_{\dot{\mathrm{p}}}^2(\Omega_f)},
\end{array}
\end{equation}
with
\begin{equation*}
\begin{array}{l}
{\mathbb{W}_{\dot{\mathrm{p}}}^1(\Omega_f)}
=\int\limits_{\Omega_f} 
\underbrace{\begin{array}{l}
\int_{0}^{t}
\Phi(t,\tau) (A(\mathrm p(t))- A(\mathrm p(\tau)))  e^{\jmath \omega \tau} (\jmath \omega I -A(\mathrm p(t)))^{-1} B(\mathrm p(\tau))   \inva\tau
\end{array}}_{\mathcal{V}_{\dot{\mathrm{p}}}^1(\jmath\omega)}
\\
{\kern 45pt}\times
\underbrace{\begin{array}{l}
\int_{0}^{t}
B^*(\mathrm p(\tau)) (\jmath \omega I -A(\mathrm p(t)))^{-*} e^{-\jmath \omega \tau}
(A(\mathrm p(t))- A(\mathrm p(\tau)))^*   
\Phi^*(t,\tau)   \inva\tau
\end{array}}_{(\mathcal{V}_{\dot{\mathrm{p}}}^1(\jmath\omega))^*}
\inva\omega,
\end{array}
\end{equation*}
and  
\begin{equation*}
\begin{array}{l}
{\mathbb{W}_{\dot{\mathrm{p}}}^2(\Omega_f)}
=
\int\limits_{\Omega_f} 
\underbrace{\begin{array}{l}
\int_{0}^{t}
{\Phi(t,\tau)}  e^{\jmath \omega \tau} (\jmath \omega I -A(\mathrm p(t)))^{-1} \dot{B}(\mathrm p(\tau)) \inva\tau
\end{array}}_{\mathcal{V}_{\dot{\mathrm{p}}}^2(\jmath\omega)}
\\
{\kern 45pt}\times
\underbrace{\begin{array}{l}
\int_{0}^{t}
\dot{B}^*(\mathrm p(\tau))(\jmath \omega I -A(\mathrm p(t)))^{-*} e^{-\jmath \omega \tau} 
{\Phi^*(t,\tau)}   \inva\tau
\end{array}}_{(\mathcal{V}_{\dot{\mathrm{p}}}^2(\jmath\omega))^*}
\inva\omega.
\end{array}
\end{equation*}

  \end{enumerate}
  
\end{definition}

\begin{remark}
Due to the bounded parameter-varying matrices $A(\mathrm p(t))$, $B(\mathrm p(t))$, $C(\mathrm p(t))$, $D(\mathrm p(t))$ and the finite-frequency spectrum $\Omega_f$, the controllability Gramians $\mathbb{W}_{\mathrm p}(\Omega_f)$ and $\hat{\mathbb{W}}_{\mathrm p}(\Omega_f)$ are bounded. 
Moreover, the controllability of LPV system ensures that the lower bounds of $\mathbb{W}_{\mathrm p}^1(\Omega_f)$ is away from zero. 

\end{remark}

\begin{remark}
$\mathbb{W}_{\dot{\mathrm{p}}}^1(\Omega_f)$ and $\mathbb{W}_{\dot{\mathrm{p}}}^2(\Omega_f)$ are 
the solutions of the following two Lyapunov inequalities:
\begin{equation}\label{Lya_W_dp}
\begin{array}{l}
-A(\mathrm p(t)) \mathbb{W}_{\dot{\mathrm{p}}}^1(\Omega_f) -\mathbb{W}_{\dot{\mathrm{p}}}^1(\Omega_f) A^T(\mathrm p(t)) 
+\dot{\mathbb{W}}_{\dot{\mathrm{p}}}^1(\Omega_f)
-\mathbb{W}_{\dot{\mathrm{p}}}^1(\Omega_f) 
-\overline{M}_1 \leq 0,
\\
-A(\mathrm p(t)) \mathbb{W}_{\dot{\mathrm{p}}}^2(\Omega_f) -\mathbb{W}_{\dot{\mathrm{p}}}^2(\Omega_f) A^T(\mathrm p(t)) 
+\dot{\mathbb{W}}_{\dot{\mathrm{p}}}^2(\Omega_f)
-\mathbb{W}_{\dot{\mathrm{p}}}^2(\Omega_f) 
-\overline{M}_2 \leq 0,
\end{array}
\end{equation}
where $\overline{M}_1$ and $\overline{M}_2$ are symmetric matrices.
\end{remark}

\begin{proof}
$\mathbb{W}_{\dot{\mathrm{p}}}^1(\Omega_f)$ and $\mathbb{W}_{\dot{\mathrm{p}}}^2(\Omega_f)$ can be written as 

\begin{equation}\label{W_dp}
\begin{array}{l}
\mathbb{W}_{\dot{\mathrm{p}}}^1(\Omega_f)
=(\int_0^t  \Phi(t,\tau) M_1(\mathrm p(\tau))\inva \tau )
(\int_0^t  M_1^T(\mathrm p(\tau)) \Phi^*(t,\tau) \inva \tau),
\\
\mathbb{W}_{\dot{\mathrm{p}}}^2(\Omega_f)
=(\int_0^t \Phi(t,\tau) M_2(\mathrm p(\tau)) \inva \tau )
(\int_0^t  M_2^T(\mathrm p(\tau)) \Phi^*(t,\tau) \inva \tau),
\end{array}
\end{equation}
where 
\[\begin{array}{l}
M_1(\mathrm p(\tau))=
(A(\mathrm p(t))- A(\mathrm p(\tau)))  e^{\jmath \omega \tau} (\jmath \omega I -A(\mathrm p(t)))^{-1} B(\mathrm p(\tau)),
\\
M_2(\mathrm p(\tau))=
e^{\jmath \omega \tau} (\jmath \omega I -A(\mathrm p(t)))^{-1} \dot{B}(\mathrm p(\tau)).
\end{array}\]

Then, its derivative satisfies
\begin{equation}
    \begin{array}{l}
\dot{\mathbb{W}}_{\dot{\mathrm{p}}}^1(\Omega_f) 
=(\int_0^t  
\frac{\partial \Phi(t,\tau)}{\partial t} M_1(\mathrm p(\tau)) \inva \tau + M(\mathrm p(t)))
(\int_0^t  M_1^T(\mathrm p(\tau)) \Phi^*(t,\tau) \inva \tau)
\\
{\kern 40pt}+(\int_0^t  \Phi(t,\tau) M_1(\mathrm p(\tau)) \inva \tau )
(M_1^T(\mathrm p(t)) 
+\int_0^t  M_1^T(\mathrm p(\tau)) \frac{\partial \Phi^*(t,\tau)}{\partial t}  \inva \tau)
\\
{\kern 40pt}=A(\mathrm p(t)) \mathbb{W}_{\dot{\mathrm{p}}}^1(\Omega_f) +\mathbb{W}_{\dot{\mathrm{p}}}^1(\Omega_f) A^T(\mathrm p(t)) 
\\
{\kern 40pt}+M(\mathrm p(t)) (\int_0^t  M_1^T(\mathrm p(\tau)) \Phi^*(t,\tau) \inva \tau)
+(\int_0^t  \Phi(t,\tau) M_1(\mathrm p(\tau)) \inva \tau )
M_1^T(\mathrm p(t))
\\
{\kern 40pt}\leq 
A(\mathrm p(t))  \mathbb{W}_{\dot{\mathrm{p}}}^1(\Omega_f)
+\mathbb{W}_{\dot{\mathrm{p}}}^1(\Omega_f) A^T(\mathrm p(t)) 
+\overline{M}_1 
+ \mathbb{W}_{\dot{\mathrm{p}}}^1(\Omega_f),
\end{array}
\end{equation}
and
\begin{equation}
\dot{\mathbb{W}}_{\dot{\mathrm{p}}}^2(\Omega_f) 
\leq 
A(\mathrm p(t))  \mathbb{W}_{\dot{\mathrm{p}}}^2(\Omega_f)
+\mathbb{W}_{\dot{\mathrm{p}}}^2(\Omega_f) A^T(\mathrm p(t)) 
+\overline{M}_2 
+ \mathbb{W}_{\dot{\mathrm{p}}}^2(\Omega_f),
\end{equation}
where $\overline{M}_1$ and $\overline{M}_2$ are the upper bounds of $M_1(\mathrm{p}(t)) M_1^T(\mathrm{p}(t))$ and $M_2(\mathrm{p}(t)) M_2^T(\mathrm{p}(t))$.

\end{proof}

Based on aforementioned, the boundedness of controllability Gramian for LPV systems is established as follows:


\begin{proposition}[Inner-connections between the controllability Gramian and state vector of LPV system] \label{gramina_LPV}
Consider the finite-frequency controllability Gramian defined in Definition \ref{con-Gramian-LPV}. Then we have   

\begin{enumerate}
    \item 
$\int_0^t \Bigl(
 \iint\limits_{\Omega_f\times \Omega_f} 
\He\Bigl(\mathcal V_{\mathrm p}^1(\jmath\omega)  
(\mathcal V_{\mathrm p}^1(\jmath\tilde\omega))^* \Bigl) 
\inva\omega \inva\tilde{\omega} \Bigr) \inva t
$
\\
$= \int_0^t \Bigl( {\mathbb{W}_{\mathrm{p}}(\Omega_f)} \Bigr)\inva t $
\\
$ +\int_0^t \Bigl( \iint\limits_{\Omega_f\times\Omega_f}
    {e^{(\jmath(\omega-\tilde\omega) t} }
    (\jmath\omega I-A(\mathrm p(t)))^{-1}  B(\mathrm p(t))
    B^*(\mathrm p(t))   (\jmath\tilde\omega I-A(\mathrm p(t)))^{-*} 
   \inva\omega\inva\tilde\omega
    \Bigr)\inva t$,

\item 
$ \int_0^t \Bigl(
 \iint\limits_{\Omega_f\times \Omega_f} 
\He\Bigl(\mathcal V_{\mathrm p}^2(\jmath\omega)  
(\mathcal V_{\mathrm p}^2(\jmath\tilde\omega))^* \Bigl) 
\inva\omega \inva\tilde{\omega} \Bigr) \inva t$
\\
$=
    \int_0^t \Bigl( {\hat{\mathbb{W}}_{\mathrm{p}}(\Omega_f)} \Bigr)\inva t $
    \\
$+\int_0^t \Bigl( \iint\limits_{\Omega_f\times\Omega_f}
    {\phi(t,0)}
    (\jmath\omega I-A(\mathrm p(t)))^{-1}  B(\mathrm p(0))
    B^*(\mathrm p(0))   (\jmath\tilde\omega I-A(\mathrm p(t)))^{-*}  {\phi^*(t,0)}
   \inva\omega\inva\tilde\omega
    \Bigr)\inva t$,

\item 
$ \int_0^t \Bigl(
 \iint\limits_{\Omega_f\times \Omega_f} 
\He\Bigl(\mathcal V_{\dot{\mathrm p}}^1(\jmath\omega)  
(\mathcal V_{\dot{\mathrm p}}^1(\jmath\tilde\omega))^* \Bigl) 
\inva\omega \inva\tilde{\omega} \Bigr) \inva t$
\\
$=
    \int_0^t \Bigl( {\mathbb{W}_{\dot{\mathrm p}}^1 (\Omega_f)} \Bigr)\inva t 
    \\
    +\int_0^t \Bigl( \iint\limits_{\Omega_f\times\Omega_f}
    \int_{0}^{t}
\Phi(t,\tau) (A(\mathrm p(t))- A(\mathrm p(\tau)))  {e^{\jmath \omega \tau}}
(\jmath \omega I -A(\mathrm p(t)))^{-1} B(\mathrm p(\tau))   \inva\tau$
\\
$\;\;\times
\int_{0}^{t}
B^*(\mathrm p(\tau))
(\jmath \tilde\omega I -A(\mathrm p(t)))^{-*}
{e^{-\jmath \tilde\omega \tau}}
(A(\mathrm p(t))- A(\mathrm p(\tau)))^*    
\Phi^*(t,\tau)  
\inva\tau
   \inva\omega\inva\tilde\omega
    \Bigr)\inva t$,

\item
$ \int_0^t \Bigl(
 \iint\limits_{\Omega_f\times \Omega_f} 
\He\Bigl(\mathcal V_{\dot{\mathrm p}}^2(\jmath\omega)  
(\mathcal V_{\dot{\mathrm p}}^2(\jmath\tilde\omega))^* \Bigl) 
\inva\omega \inva\tilde{\omega} \Bigr) \inva t$
\\
$=
\int_0^t \Bigl( {\mathbb{W}_{\dot{\mathrm p}}^2 (\Omega_f)} \Bigr)\inva t 
$
\\
$+\int_0^t \Bigl( \iint\limits_{\Omega_f\times\Omega_f}
   \int_{0}^{t}
{\Phi(t,\tau)}  {e^{\jmath \omega \tau}} (\jmath \omega I -A(\mathrm p(t)))^{-1} \dot{B}(\mathrm p(\tau)) \inva\tau$
\\
$\times
\int_{0}^{t}
\dot{B}^*(\mathrm p(\tau))
(\jmath \tilde\omega I -A(\mathrm p(t)))^{-*}
{e^{-\jmath \tilde\omega \tau}}
{\Phi^*(t,\tau)}   
\inva\tau
\inva\omega\inva\tilde\omega
    \Bigr)\inva t.$

\end{enumerate}

In particular, for BIBS/UAS LPV system, we have 
\begin{enumerate}
    \item 
$ \lim\limits_{t\to \infty} \frac 1t {\int_0^t  \Bigl(
 \iint\limits_{\Omega_f\times \Omega_f} 
\He\Bigl(\mathcal V_{\mathrm p}^1(\jmath\omega)
(\mathcal V_{\mathrm p}^1(\jmath\tilde\omega))^* \Bigl) 
\inva\omega \inva\tilde{\omega} \Bigr) \inva t}
={\mathbb{W}_{\mathrm{p}}(\Omega_f)},$

\item
$\lim\limits_{t\to \infty} \frac 1t {\int_0^t  \Bigl(
 \iint\limits_{\Omega_f\times \Omega_f} 
\He\Bigl(\mathcal V_{\mathrm p}^2(\jmath\omega)
(\mathcal V_{\mathrm p}^2(\jmath\tilde\omega))^* \Bigl) 
\inva\omega \inva\tilde{\omega} \Bigr) \inva t}
$
\\
$    =\left\{\begin{array}{l}
     \begin{smallmatrix}
     \hat{\mathbb{W}}_{\mathrm{p}}(\Omega_f)
     +\iint\limits_{\Omega_f\times\Omega_f}
    \phi(t,0)
    (\jmath\omega I-A(\mathrm p(t)))^{-1}  B(\mathrm p(0))
    B^*(\mathrm p(0))   (\jmath\tilde\omega I-A(\mathrm p(t)))^{-*}  \phi^*(t,0)
   \inva\omega\inva\tilde\omega
    \Bigr)\inva t,
    \end{smallmatrix}~{\rm BIBS},
\\
0,~{\rm UAS},
    \end{array}\right.
$
    
\item 
$\lim\limits_{t\to \infty} \frac 1t {\int_0^t  \Bigl(
 \iint\limits_{\Omega_f\times \Omega_f} 
\He\Bigl(\mathcal V_{\dot{\mathrm{p}}}^1(\jmath\omega)
(\mathcal V_{\dot{\mathrm{p}}}^1(\jmath\tilde\omega))^* \Bigl) 
\inva\omega \inva\tilde{\omega} \Bigr) \inva t}=
{ 
\mathbb{W}_{\dot{\mathrm{p}}}^1(\Omega_f)
},$

\item 
$ \lim\limits_{t\to \infty} \frac 1t {\int_0^t  \Bigl(
 \iint\limits_{\Omega_f\times \Omega_f} 
\He\Bigl(\mathcal V_{\dot{\mathrm{p}}}^2(\jmath\omega)
(\mathcal V_{\dot{\mathrm{p}}}^2(\jmath\tilde\omega))^* \Bigl) 
\inva\omega \inva\tilde{\omega} \Bigr) \inva t}
={ 
\mathbb{W}_{\dot{\mathrm{p}}}^2(\Omega_f)},$

\item 
$ \lim\limits_{t\to \infty} \frac 1t {\int_0^t  \Bigl(
 \iint\limits_{\Omega_f\times \Omega_f} 
\He\Bigl(\mathcal V_{{\mathrm{p}}}^1(\jmath\omega)
(\mathcal V_{{\mathrm{p}}}^2(\jmath\tilde\omega))^* \Bigl) 
\inva\omega \inva\tilde{\omega} \Bigr) \inva t}
=0,$

\item 
$ \lim\limits_{t\to \infty} \frac 1t {\int_0^t  \Bigl(
 \iint\limits_{\Omega_f\times \Omega_f} 
\He\Bigl(\mathcal V_{{\mathrm{p}}}^1(\jmath\omega)
(\mathcal V_{\dot{\mathrm{p}}}^1(\jmath\tilde\omega))^* \Bigl) 
\inva\omega \inva\tilde{\omega} \Bigr) \inva t}
=0,$

\item 
$ \lim\limits_{t\to \infty} \frac 1t {\int_0^t  \Bigl(
 \iint\limits_{\Omega_f\times \Omega_f} 
\He\Bigl(\mathcal V_{{\mathrm{p}}}^1(\jmath\omega)
(\mathcal V_{\dot{\mathrm{p}}}^2(\jmath\tilde\omega))^* \Bigl) 
\inva\omega \inva\tilde{\omega} \Bigr) \inva t}
=0,$

\item 
$ \lim\limits_{t\to \infty} \frac 1t {\int_0^t  \Bigl(
 \iint\limits_{\Omega_f\times \Omega_f} 
\He\Bigl(\mathcal V_{{\mathrm{p}}}^2(\jmath\omega)
(\mathcal V_{\dot{\mathrm{p}}}^1(\jmath\tilde\omega))^* \Bigl) 
\inva\omega \inva\tilde{\omega} \Bigr) \inva t}
=0,$

\item 
$ \lim\limits_{t\to \infty} \frac 1t {\int_0^t  \Bigl(
 \iint\limits_{\Omega_f\times \Omega_f} 
\He\Bigl(\mathcal V_{{\mathrm{p}}}^2(\jmath\omega)
(\mathcal V_{\dot{\mathrm{p}}}^2(\jmath\tilde\omega))^* \Bigl) 
\inva\omega \inva\tilde{\omega} \Bigr) \inva t}
=0,$

\item 
$ \lim\limits_{t\to \infty} \frac 1t {\int_0^t  \Bigl(
 \iint\limits_{\Omega_f\times \Omega_f} 
\He\Bigl(\mathcal V_{\dot{\mathrm{p}}}^1(\jmath\omega)
(\mathcal V_{\dot{\mathrm{p}}}^2(\jmath\tilde\omega))^* \Bigl) 
\inva\omega \inva\tilde{\omega} \Bigr) \inva t}
=0,$

\end{enumerate}

where $\mathcal V_{\mathrm p}^1(\jmath\omega)$, $\mathcal V_{\mathrm p}^2(\jmath\omega)$, $\mathcal V_{\dot{\mathrm p}}^1(\jmath\omega)$ and $\mathcal V_{\dot{\mathrm p}}^2(\jmath\omega)$ are defined as \eqref{V-LPV}.

\end{proposition}

\begin{proof}
This conclusion can be proven by leveraging the boundedness properties of BIBS stable/UAS LPV system in Lemma \ref{boundness:BIBS} and Lemma \ref{boundness:UAS}.
\end{proof}

\begin{remark}
    The controllability of LPV systems plays a crucial role in ensuring the boundedness of the state vector components. 
    Moreover, controllability is fundamental for generalizing gKYP lemma form the LTI system framework to the LPV systems setting in the next subsection.
\end{remark}

\subsection{Time domain interpretation of finite-frequency input signal}

Firstly, let us define the gap between a matrix and a finite-frequency range.


\begin{definition}
[Gap between the system matrix and finite-frequency range]
Given a parameter dependent matrix $A(\mathrm{p}(t))$ and a finite-frequency range $\Omega_f$, for all $t> 0$,
\begin{equation}\label{gap}
\begin{array}{l}
    \bigl(d(A(\mathrm{p}(t)),\Omega_f)\bigl)^2
   =\left\{\begin{array}{l}
         0, {\kern 189pt} {\rm if}~ 
\left[\begin{smallmatrix}
     A^*(\mathrm{p}(t)) & I
       \end{smallmatrix}\right]
       (\Psi_{f} \otimes I)
       \left[\begin{smallmatrix}
     A(\mathrm{p}(t)) \\
     I
         \end{smallmatrix}\right] \geq 0,
         \\
         \lambda_{\rm max}\Bigl(-\left[\begin{smallmatrix}
     A^*(\mathrm{p}(t)) & I
       \end{smallmatrix}\right]
       (\Psi_{f} \otimes I)
       \left[\begin{smallmatrix}
     A(\mathrm{p}(t)) \\
     I
         \end{smallmatrix}\right] \Bigl),~
         {\rm if}~
\left[\begin{smallmatrix}
     A^*(\mathrm{p}(t)) & I
       \end{smallmatrix}\right]
       (\Psi_{f} \otimes I)
       \left[\begin{smallmatrix}
     A(\mathrm{p}(t)) \\
     I
         \end{smallmatrix}\right] < 0,
    \end{array}\right.
    \end{array}
\end{equation}
will be referred as to the gap between the matrix $A(\mathrm{p}(t))$ and finite-frequency range $\Omega_f$.    
\end{definition}


Then, incorporating an enlarged finite-frequency interval, the non-negativity of IQC for LPV systems can be guaranteed as follows:


\begin{theorem}
[Time domain interpretation under the excitation of finite-frequency input signal, LPV system]
\label{Th_LPV_BIBS/UAS} 
Consider LPV system \eqref{LPV-time} with $(A(\mathrm p(t)), B(\mathrm p(t)))$ controllable and assume LPV system \eqref{LPV-time} is BIBS stable. If the input signal $u(t)$ is a finite-frequency signal with respect to $\Omega_f$. Then there exist an enlarged frequency range ${\Omega}_{f+\delta}$, such that system state and its derivatives satisfy the following $({\Omega}_{f+\delta})$-IQC: 
 \begin{equation}\label{s_IQC_LPV}
 \int_0^{\infty} 
    \He\Bigl(
    \begin{bmatrix} \dot x^*(t) & x^*(t) \end{bmatrix} 
    \Psi_{f+\delta} 
    \begin{bmatrix} \dot x(t)\\ x(t) \end{bmatrix}
    \Bigl)
    \inva t \geq 0,
	\end{equation}
where ${\Psi}_{f+\delta}$ is the companion matrix with an enlarged frequency range ${\Omega}_{f+\delta}$ can be constructed as: 
\begin{equation}\label{new_range}
\begin{array}{l}
\Omega_l \subset   \Omega_{l+\delta}
:=\Omega_l\oplus\Omega_\delta
=[-\sqrt{\varpi_l^2+\delta^2}, \sqrt{\varpi_l^2+\delta^2}],
\\
\Omega_m \subset \Omega_{m+\delta}
:=\Omega_l\oplus\Omega_\delta
=[-\varpi_1-\frac{\sqrt{(\varpi_2-\varpi_1)^2 +4 \delta^2}}{2},-\varpi_2+\frac{\sqrt{(\varpi_2-\varpi_1)^2 +4 \delta^2}}{2}]
\\
{\kern 130pt}\cup[\varpi_1-\frac{\sqrt{(\varpi_2-\varpi_1)^2 +4 \delta^2}}{2},\varpi_2+\frac{\sqrt{(\varpi_2-\varpi_1)^2 +4 \delta^2}}{2}],
\\
\Omega_h \subset \Omega_{h+\delta}
:=\Omega_h\oplus\Omega_\delta
=(-\infty, -\sqrt{\varpi_h^2-\delta^2}] \cup [\sqrt{\varpi_h^2-\delta^2},+\infty),
\end{array}
\end{equation}
with $\delta$ satisfying
    \begin{equation}\label{delta}
    \begin{array}{l}
\delta^2 
\geq 
\bigl(d(A(\mathrm p(t)),\Omega_f)\bigl)^2  
    \Bigl(
    -\tr\bigl(
    \hat{\mathbb W}_{\mathrm p}(\Omega_f)\bigl) 
+
\tr\bigl(\mathbb W_{\dot {\mathrm p}}(\Omega_f)\bigl)
    \Bigl)
    \times 
    \Bigl(
    \tr(\mathbb W_{\mathrm p}(\Omega_f)
    \Bigr)^{-1}.
\end{array}
\end{equation}

In particular, for UAS LPV systems, a more elegant  estimate of the admissible values of $\delta$ can be given as: 
\begin{equation}\label{delta-UAS}
    \begin{array}{l}
    \delta^2 \geq   
    \bigl(d(A(\mathrm p(t)),\Omega_f)\bigl)^2  
    \tr\bigl(\mathbb W_{\dot {\mathrm p}}(\Omega_f)\bigl) 
    \times
    \Bigl( 
    \tr(\mathbb W_{\mathrm p}(\Omega_f))
    \Bigr)^{-1}.
    \end{array}
\end{equation}

\end{theorem}


\begin{proof}

As well known that inequality \eqref{s_IQC_LPV} holds if and only if the following formula holds:
\begin{equation}
    \tr \Bigl( 
    \int_0^{\infty} 
    \He\Bigl(
    \begin{bmatrix} \dot x(t) & x(t) \end{bmatrix} 
    \Psi_{f+\delta} 
    \begin{bmatrix} \dot x^*(t)\\ x^*(t) \end{bmatrix}
    \Bigl)
    \inva t
    \Bigl) \geq 0,
\end{equation}
with
\begin{equation}\label{Th1-eqn3}
\begin{array}{l}
\tr\bigg( \He \Bigl(
    \begin{bmatrix} \dot x(t)& x(t) \end{bmatrix} 
    \Psi_{f+\delta}
    \begin{bmatrix} \dot x^*(t)\\ x^*(t) \end{bmatrix} 
     \Bigl) \bigg)
\\
= 
\tr\bigg(  \He\Bigl(
\begin{bmatrix} 
    \dot x(t)& x(t) 
    \end{bmatrix} 
    \begin{bmatrix}
        \Psi_f^{11} & \Psi_f^{12} \\
        \Psi_f^{21} & \Psi_f^{22}+\delta^2
    \end{bmatrix} 
    \begin{bmatrix} 
    \dot x^*(t)\\ 
    x^*(t) 
    \end{bmatrix} 
\Bigl)\bigg)
\\
=
\iint\limits_{\Omega_f\times\Omega_f}  
   \tr\bigg(  \He\Bigl(
   \Psi_f^{11}
   \Bigl(\jmath\omega \mathcal{V}_{\mathrm{p}}^1(\jmath\omega)
   + A \mathcal{V}_{\mathrm{p}}^2(\jmath\omega)
   + A \mathcal{V}_{\dot{\mathrm{p}}}^1(\jmath\omega)
   + A \mathcal{V}_{\dot{\mathrm{p}}}^2(\jmath\omega)
   \Bigl) \mathcal U(\jmath\omega) 
   \\
   \times
   \mathcal U^*(\jmath\tilde\omega) 
\Bigl(\jmath\tilde\omega \mathcal{V}_{\mathrm{p}}^1(\jmath\tilde\omega)
+ A \mathcal{V}_{\mathrm{p}}^2(\jmath\tilde\omega)
+ A \mathcal{V}_{\dot{\mathrm{p}}}^1(\jmath\tilde\omega)
+ A \mathcal{V}_{\dot{\mathrm{p}}}^2(\jmath\tilde\omega)
\Bigl)^* 
\Bigl)\bigg)
    \inva \omega   \inva \tilde\omega 
\\
+
\iint\limits_{\Omega_f\times\Omega_f}  
   \tr\bigg(  \He\Bigl(
   \Psi_f^{12}
   \Bigl(\jmath\omega \mathcal{V}_{\mathrm{p}}^1(\jmath\omega)
   + A \mathcal{V}_{\mathrm{p}}^2(\jmath\omega)
   + A \mathcal{V}_{\dot{\mathrm{p}}}^1(\jmath\omega)
    + A \mathcal{V}_{\dot{\mathrm{p}}}^2(\jmath\omega)
   \Bigl) \mathcal U(\jmath\omega) 
   \\
   \times
   \mathcal U^*(\jmath\tilde\omega) 
\Bigl(\mathcal{V}_{\mathrm{p}}^1(\jmath\tilde\omega)
+ \mathcal{V}_{\mathrm{p}}^2(\jmath\tilde\omega))
+ \mathcal{V}_{\dot{\mathrm{p}}}^1(\jmath\tilde\omega)
+ \mathcal{V}_{\dot{\mathrm{p}}}^2(\jmath\tilde\omega)
\Bigl)^* 
\Bigl)\bigg)
    \inva \omega   \inva \tilde\omega 
\\
+
\iint\limits_{\Omega_f\times\Omega_f}  
   \tr\bigg(  \He\Bigl(
   \Psi_f^{21}
   \Bigl( \mathcal{V}_{\mathrm{p}}^1(\jmath\omega)
   + \mathcal{V}_{\mathrm{p}}^2(\jmath\omega)
   + \mathcal{V}_{\dot{\mathrm{p}}}^1(\jmath\omega)
   + \mathcal{V}_{\dot{\mathrm{p}}}^2(\jmath\omega)
   \Bigl) 
   \mathcal U(\jmath\omega)
      \\
   \times
   \mathcal U^*(\jmath\tilde\omega) 
\Bigl(\jmath\tilde\omega \mathcal{V}_{\mathrm{p}}^1(\jmath\tilde\omega)
+ A \mathcal{V}_{\mathrm{p}}^2(\jmath\tilde\omega)
+ A \mathcal{V}_{\dot{\mathrm{p}}}^1(\jmath\tilde\omega)
+ A \mathcal{V}_{\dot{\mathrm{p}}}^2(\jmath\tilde\omega)
\Bigl)^* 
\Bigl)\bigg)
    \inva \omega   \inva \tilde\omega 
\\
+
\iint\limits_{\Omega_f\times\Omega_f}  
   \tr\bigg(  \He\Bigl(
   (\Psi_f^{22}+\delta^2)
   \Bigl(\mathcal{V}_{\mathrm{p}}^1(\jmath\omega)
   + \mathcal{V}_{\mathrm{p}}^2(\jmath\omega)
   + \mathcal{V}_{\dot{\mathrm{p}}}^1(\jmath\omega)
   + \mathcal{V}_{\dot{\mathrm{p}}}^2(\jmath\omega)
   \Bigl) \mathcal U(\jmath\omega)
      \\
   \times
   \mathcal U^*(\jmath\tilde\omega) 
\Bigl(\mathcal{V}_{\mathrm{p}}^1(\jmath\tilde\omega)
+ \mathcal{V}_{\mathrm{p}}^2(\jmath\tilde\omega)
+ \mathcal{V}_{\dot{\mathrm{p}}}^1(\jmath\tilde\omega)
+ \mathcal{V}_{\dot{\mathrm{p}}}^2(\jmath\tilde\omega)
\Bigl)^* 
\Bigl)\bigg)
    \inva \omega   \inva \tilde\omega.
\end{array}
\end{equation}

Integrating equation \eqref{Th1-eqn3} from 0 to $\infty$ and combining Proposition \ref{gramina_LPV}, we obtain
\begin{equation}\label{proof_int_IQC}
\begin{array}{l}
\int_0^\infty
\tr\bigg(  \He\Bigl(
\begin{bmatrix} 
    \dot x(t)& x(t) 
    \end{bmatrix} 
    \Psi_{f+\delta} 
    \begin{bmatrix} 
    \dot x^*(t)\\ 
    x^*(t) 
    \end{bmatrix} 
\Bigl)\bigg) 
\inva t
\\
=\int_0^\infty 
\tr \bigg(
\int\limits_{\Omega_f}  
   (\begin{bmatrix} 
    \jmath\omega  &     1    
    \end{bmatrix}  \Psi_{f}
    \begin{bmatrix}    
    \jmath\omega &     1
    \end{bmatrix}^*
    +\delta^2)
      \mathcal V_{\mathrm{p}}^1(\jmath\omega) 
     \mathcal U(\jmath\omega)
   \mathcal U^*(\jmath\omega) 
   (\mathcal V_{\mathrm{p}}^1 (\jmath\omega))^*
    \inva \omega
     \bigg)
     \inva t
    \\
+
\int_0^\infty 
\tr\bigg(
\begin{bmatrix}
    A^*(\mathrm p(t)) & I
\end{bmatrix}
(\Psi_f \otimes I)
\begin{bmatrix}
   A(\mathrm p(t))\\
   I
\end{bmatrix}
\int\limits_{\Omega_f}
    \mathcal V_{\dot{\mathrm{p}}}^1 (\jmath\omega) 
    \mathcal U(\jmath\omega)
   \mathcal U^*(\jmath\omega) 
    (\mathcal V_{\dot{\mathrm{p}}}^1 (\jmath\omega))^*
    \inva \omega 
    \bigg)
 \inva t
 \\
 +
\int_0^\infty 
\tr\bigg(
\begin{bmatrix}
    A^*(\mathrm p(t)) & I
\end{bmatrix}
(\Psi_f \otimes I)
\begin{bmatrix}
   A(\mathrm p(t))\\
   I
\end{bmatrix}
\int\limits_{\Omega_f}
\mathcal V_{\dot{\mathrm{p}}}^2 (\jmath\omega) 
    \mathcal U(\jmath\omega)
   \mathcal U^*(\jmath\omega) 
    (\mathcal V_{\dot{\mathrm{p}}}^2 (\jmath\omega))^*
    \inva \omega 
    \bigg)
 \inva t
\\
    +
\int_0^\infty 
\tr\bigg(
\begin{bmatrix}
    A^*(\mathrm p(t)) & I
\end{bmatrix}
(\Psi_f \otimes I)
\begin{bmatrix}
   A(\mathrm p(t))\\
   I
\end{bmatrix}
\iint\limits_{\Omega_f\times\Omega_f}
\He\Bigl(
    \mathcal V_{\mathrm{p}}^2 (\jmath\omega) 
    \mathcal U(\jmath\omega)
   \mathcal U^*(\jmath\tilde\omega) 
   (\mathcal V_{\mathrm{p}}^2 (\jmath\tilde\omega))^*
    \Bigl)
    \inva \omega \inva \tilde \omega
    \bigg)
 \inva t
 \\
+
    \int_0^\infty 
     \tr \bigg(
   \iint\limits_{\Omega_f\times\Omega_f, \omega\neq \tilde\omega}  
   \He\Bigl(
   \bigl(\begin{bmatrix} 
    \jmath\omega  &     1    
    \end{bmatrix}  \Psi_{f}
    \begin{bmatrix}    
    \jmath\tilde\omega &     1
    \end{bmatrix}^*+\delta^2
    \bigl)
   \mathcal V_{\mathrm{p}}^1(\jmath\omega) 
     \mathcal U(\jmath\omega)
   \mathcal U^*(\jmath\tilde\omega) 
   (\mathcal V_{\mathrm{p}}^1 (\jmath\tilde\omega))^*
    \Bigl)
    \inva \omega \inva \tilde \omega
    \bigg)
    \inva t
    \\
    +
\int_0^\infty 
\tr\bigg(
\begin{bmatrix}
    A^*(\mathrm p(t)) & I
\end{bmatrix}
(\Psi_f \otimes I)
\begin{bmatrix}
   A(\mathrm p(t))\\
   I
\end{bmatrix}
\\
\times
\iint\limits_{\Omega_f\times\Omega_f,\omega\neq\tilde\omega}
\He \Bigl(   
    \mathcal V_{\dot{\mathrm{p}}}^1 (\jmath\omega) 
    \mathcal U(\jmath\omega)
   \mathcal U^*(\jmath\tilde\omega) 
    (\mathcal V_{\dot{\mathrm{p}}}^1 (\jmath\tilde\omega))^*
    +\mathcal V_{\dot{\mathrm{p}}}^2 (\jmath\omega) 
    \mathcal U(\jmath\omega)
   \mathcal U^*(\jmath\tilde\omega) 
   (\mathcal V_{\dot{\mathrm{p}}}^2 (\jmath\tilde\omega))^*
    \Bigl)
    \inva \omega \inva \tilde \omega
    \bigg)
 \inva t
    \\
+ 
\int_0^\infty 
\tr\bigg(
\bigl(
\begin{bmatrix}
    A^*(\mathrm p(t)) & I
\end{bmatrix}
(\Psi_f \otimes I)
\begin{bmatrix}
   A(\mathrm p(t))\\
   I
\end{bmatrix}+\delta ^2 I
\bigl)
\\
\times
\iint\limits_{\Omega_f\times\Omega_f}
 \He\Bigl(
    \mathcal V_{\mathrm{p}}^2 (\jmath\omega) 
    \mathcal U(\jmath\omega)
   \mathcal U^*(\jmath\tilde\omega) 
    \bigl(\mathcal V_{\dot{\mathrm{p}}}^1 (\jmath\tilde\omega)
    +\mathcal V_{\dot{\mathrm{p}}}^2 (\jmath\tilde\omega) \bigl)^* 
    +
    \mathcal V_{\dot{\mathrm{p}}}^1(\jmath\omega)
    \mathcal U(\jmath\omega)
   \mathcal U^*(\jmath\tilde\omega) 
    \bigl(\mathcal V_{\mathrm{p}}^2 (\jmath\omega) 
    \mathcal V_{\dot{\mathrm{p}}}^2 (\jmath\omega)
    \bigl)^*
\\
+
\mathcal V_{\dot{\mathrm{p}}}^2(\jmath\omega)
\mathcal U(\jmath\omega)
   \mathcal U^*(\jmath\tilde\omega) 
   \bigl(\mathcal V_{\mathrm{p}}^2 (\jmath\omega) 
    \mathcal V_{\dot{\mathrm{p}}}^1 (\jmath\omega)
    \bigl)^*
    \Bigl)
    \inva \omega \inva \tilde \omega
    \bigg)
 \inva t
    \\
     +
\int_0^\infty 
\tr\bigg( 
\He \Bigl(
\bigl(
\begin{bmatrix}
    A^*(\mathrm p(t)) & I
\end{bmatrix}
(\Psi_f\otimes I)
\begin{bmatrix}
    \jmath\omega I\\
    I
\end{bmatrix}+\delta ^2 I \bigl)
 \\
\times
\iint\limits_{\Omega_f\times\Omega_f}
    \mathcal V_{\mathrm{p}}^1 (\jmath\omega) 
     \mathcal U(\jmath\omega)
   \mathcal U^*(\jmath\tilde\omega) 
    (\mathcal V_{\mathrm{p}}^2 (\jmath\tilde\omega)
    +\mathcal V_{\dot{\mathrm{p}}}^1 (\jmath\tilde\omega)
    +\mathcal V_{\dot{\mathrm{p}}}^2 (\jmath\tilde\omega)
    )^* 
     \Bigl)
     \inva \omega \inva \tilde \omega
    \bigg)
 \inva t
 \\
+ 
\int_0^\infty 
\tr\bigg(
\He\Bigl(
\bigl( \begin{bmatrix}
    -\jmath\tilde\omega I &     I
\end{bmatrix}
(\Psi_f\otimes I)
\begin{bmatrix}
    A(\mathrm p(t)) \\
    I
\end{bmatrix} 
+\delta ^2 I \bigl)
\\
\times
\iint\limits_{\Omega_f\times\Omega_f}
    \bigl(\mathcal V_{\mathrm{p}}^2 (\jmath\omega) 
    +\mathcal V_{\dot{\mathrm{p}}}^1 (\jmath\omega) 
    +\mathcal V_{\dot{\mathrm{p}}}^2 (\jmath\omega)\bigl) 
      \mathcal U(\jmath\omega)
   \mathcal U^*(\jmath\tilde\omega) 
    (\mathcal V_{\mathrm{p}}^1 (\jmath\tilde\omega))^*
    \Bigl) 
    \inva \omega \inva \tilde \omega
  \bigg)
 \inva t.
    \end{array}
\end{equation}


According to Proposition \ref{gramina_LPV}, the last five terms of \eqref{proof_int_IQC} are uniformly bounded, then \eqref{proof_int_IQC} can be written as 
\begin{equation}\label{proof_int_IQC_2}
\begin{array}{l}
\int_0^\infty
\tr\bigg(  \He\Bigl(
\begin{bmatrix} 
    \dot x(t)& x(t) 
    \end{bmatrix} 
    \Psi_{f+\delta} 
    \begin{bmatrix} 
    \dot x^*(t)\\ 
    x^*(t) 
    \end{bmatrix} 
\Bigl)\bigg) 
\inva t
\\
\geq
\int_0^\infty 
\tr \bigg(
\int\limits_{\Omega_f}  
   \delta^2
      \mathcal V_{\mathrm{p}}^1(\jmath\omega) 
     \mathcal U(\jmath\omega)
   \mathcal U^*(\jmath\omega) 
   (\mathcal V_{\mathrm{p}}^1 (\jmath\omega))^*
    \inva \omega
     \bigg)
     \inva t
    \\
-
\int_0^\infty 
\tr\bigg(
\begin{bmatrix}
    A^*(\mathrm p(t)) & I
\end{bmatrix}
(\Psi_f \otimes I)
\begin{bmatrix}
   A(\mathrm p(t))\\
   I
\end{bmatrix}
\int\limits_{\Omega_f}
     \mathcal V_{\mathrm{p}}^2 (\jmath\omega) 
    \mathcal U(\jmath\omega)
   \mathcal U^*(\jmath\omega) 
   (\mathcal V_{\mathrm{p}}^2 (\jmath\omega))^*
   \inva\omega
     
    \bigg)
 \inva t
     \\
+
\int_0^\infty 
\tr\bigg(
\begin{bmatrix}
    A^*(\mathrm p(t)) & I
\end{bmatrix}
(\Psi_f \otimes I)
\begin{bmatrix}
   A(\mathrm p(t))\\
   I
\end{bmatrix}
\int\limits_{\Omega_f}
    \mathcal V_{\dot{\mathrm{p}}}^1 (\jmath\omega) 
     \mathcal U(\jmath\omega)
   \mathcal U^*(\jmath\omega) (\mathcal V_{\dot{\mathrm{p}}}^1 (\jmath\omega))^*
    \inva \omega 
    \bigg)
 \inva t
\\
+
\int_0^\infty 
\tr\bigg(
\begin{bmatrix}
    A^*(\mathrm p(t)) & I
\end{bmatrix}
(\Psi_f \otimes I)
\begin{bmatrix}
   A(\mathrm p(t))\\
   I
\end{bmatrix}
\int\limits_{\Omega_f}
   \mathcal V_{\dot{\mathrm{p}}}^2 (\jmath\omega) 
    \mathcal U(\jmath\omega)
   \mathcal U^*(\jmath\omega) (\mathcal V_{\dot{\mathrm{p}}}^2 (\jmath\omega))^*
    \inva \omega 
    \bigg)
 \inva t,
    \end{array}
\end{equation}
hence, for $\delta$ satisfies inequality \eqref{delta}, then \eqref{proof_int_IQC_2} is non-negative, and the $\Omega_{f+\delta}$-IQC \eqref{s_IQC_LPV} holds.

In particular, for UAS LPV system, the second and fifth terms in \eqref{proof_int_IQC_2} tend to zero, so \eqref{delta-UAS} holds.

\end{proof}

\begin{remark}

For LPV systems whose poles location are constrained in a finite-frequency range $\Omega_f$, it also means that the following matrix inequality holds for all admissible parameter trajectories $\mathrm p(t)$
\[\begin{bmatrix}
   A(\mathrm p(t)) \\ I 
\end{bmatrix}^* (\Psi_{f} \otimes I)
     \begin{bmatrix}
     A(\mathrm p(t)) \\ I
        \end{bmatrix} 
        \geq 0,
        \]
the $\Omega_f$-IQC is guaranteed to be non-negative.

\end{remark}

\begin{remark}
The time domain interpretation presented in the Theorem \ref{Th_LPV_BIBS/UAS} provide us a further insight for characterizing the dynamic behavior of LPV system under the excitation of finite-frequency input signal. Although there are out-of-band frequency components due to the inter-modulation between the input signal and the time-varying parameters, it is still feasible to generate a clear picture of the dynamic behavior, with regard to the integral quadratic function of system state and its derivative.

\end{remark}

\subsection{New extension of gKYP lemma}

Based on the non-negative ${\Omega}_{f+\delta}$-IQC indicated in Theorem \ref{Th_LPV_BIBS/UAS}, we now in position to give the analysis condition for LPV systems, which is an extension of gKYP lemma.

\begin{theorem}\label{Extension-LPV-gKYP} 
Consider a LPV system \eqref{LPV-time} with zero initial condition and let the result of Theorem \ref{Th_LPV_BIBS/UAS} holds. 
If there exist parameter-dependent
Hermitian matrices $P(\mathrm p(t))$, $Q(\mathrm p(t))>0$, with $P(\mathrm p(t))$ being
differentiable, and the following inequality holds
 \begin{equation}
 \label{eqn_UASLPV-1}
\begin{split}
               \begin{bmatrix}
				A(\mathrm p(t)) &  B(\mathrm p(t))\\
				I& 0 \\
			 \end{bmatrix}^*
             \bigl(\Theta \otimes P(\mathrm p(t)) + \Theta_d \otimes \dot P(\mathrm p(t))+{\Psi_{f+\delta}} \otimes Q(\mathrm p(t))\bigl)
              \begin{bmatrix}
				A(\mathrm p(t)) &  B(\mathrm p(t))\\
				I & 0\\
			\end{bmatrix} &
            \\ 
        +\begin{bmatrix}
				C (\mathrm p(t))   & D (\mathrm p(t))  \\
				0 & I \\
			\end{bmatrix}^*\Pi \begin{bmatrix}
				C (\mathrm p(t))   & D (\mathrm p(t))  \\
				0 & I \\
			\end{bmatrix}   &\leq 0,
            \end{split}
	\end{equation}
for some performance index matrix $\Pi \in  \mathbb R^{(q+p)\times (q+p)}$, $\Theta,\Theta_d$ as defined in Lemma \ref{lemma_LPV_FF}, and {$\Psi_{f+\delta}$ as in \eqref{new_range}}, 
then the LPV system \eqref{LPV-time} satisfies the input-output performance \eqref{IO-Performance-TD}/\eqref{IO-Performance-FD}.

\end{theorem}

\begin{proof}	
Multiplying the inequality \eqref{eqn_UASLPV-1} by 
$\begin{bmatrix}
	x(t)\\
	u(t)\\
\end{bmatrix}$ 
from the right and by its conjugate transpose from the left, we obtain
\begin{equation*}
    \begin{split}
\begin{bmatrix}
		x(t)\\
		u(t)
\end{bmatrix}^* \begin{bmatrix}
	C(\mathrm p(t)) & D(\mathrm p(t))\\
					0 & I
\end{bmatrix}^*\Pi\begin{bmatrix}
		 C(\mathrm p(t)) &  D(\mathrm p(t))\\
			0 & I
		\end{bmatrix}\begin{bmatrix}
					x(t)\\
					u(t)
				\end{bmatrix} &\\
+ \frac {d}{dt}(x^*(t) P(\mathrm p(t)) x(t)) 
+\begin{bmatrix}
                   \dot x(t)\\
                          x(t)
    \end{bmatrix}^* {\Psi_{f+\delta}} \otimes Q (\mathrm p(t)) 
    \begin{bmatrix}
                       \dot x(t)\\
                          x(t)    
                       \end{bmatrix}  &\leq 0.
    \end{split}
\end{equation*}
Integrating from $t=0$ to $\infty$ and according to Theorem \ref{Th_LPV_BIBS/UAS}, we obtain    
\begin{equation*}
\begin{split}
        &\int_0^{\infty}  \begin{bmatrix}
					y(t)\\
					u(t)
				\end{bmatrix}^* \Pi
				\begin{bmatrix}
					y(t)\\
					u(t)
				\end{bmatrix}\inva t  
                \\
                &\leq 
                x(0)P(\mathrm{p}(0))x(0)-x(\infty)P(\mathrm{p}(\infty))x(\infty) 
    - \int_0^{\infty}      
    \begin{bmatrix}
    \dot x(t)\\
     x(t)
    \end{bmatrix}^* 
    {\Psi_{f+\delta}} \otimes Q (\mathrm p(t)) \begin{bmatrix}
                       \dot x(t)\\
                          x(t)    
                       \end{bmatrix} \inva t
\\
&\leq 0.
\end{split}
\end{equation*}

	\end{proof}

\begin{remark}
The underlying difference between the performance-revealing inequality \eqref{eqn_UASLPV-1} in Theorem \ref{Extension-LPV-gKYP} and the performance-revealing inequality \eqref{L2_FF_LPV} in Lemma \ref{lemma_LPV_FF} lies in the frequency selective weight matrix. The major contribution of Theorem \ref{Extension-LPV-gKYP} is that an enlarged frequency interval and 
new frequency selective weight matrix imply direct validity of finite-frequency constrained input-output performance criteria for finite-frequency signals and, thus, meets the demand of finite-frequency analysis problems.
\end{remark}

\begin{remark}

The inter-relationship between the input signals and the input-output performance involved in LPV-KYP lemma \ref{lemma_LPV_EF}, and the inter-relationship between the input signals, time domain interpretation as well as the input-output performance involved in LPV-gKYP lemma \ref{lemma_LPV_FF} and Theorem \ref{Extension-LPV-gKYP} by the Fig. \ref{fig:IO-LPV-diagram}.

\begin{figure}[htp]
    \centering
    \includegraphics[width=1\linewidth]{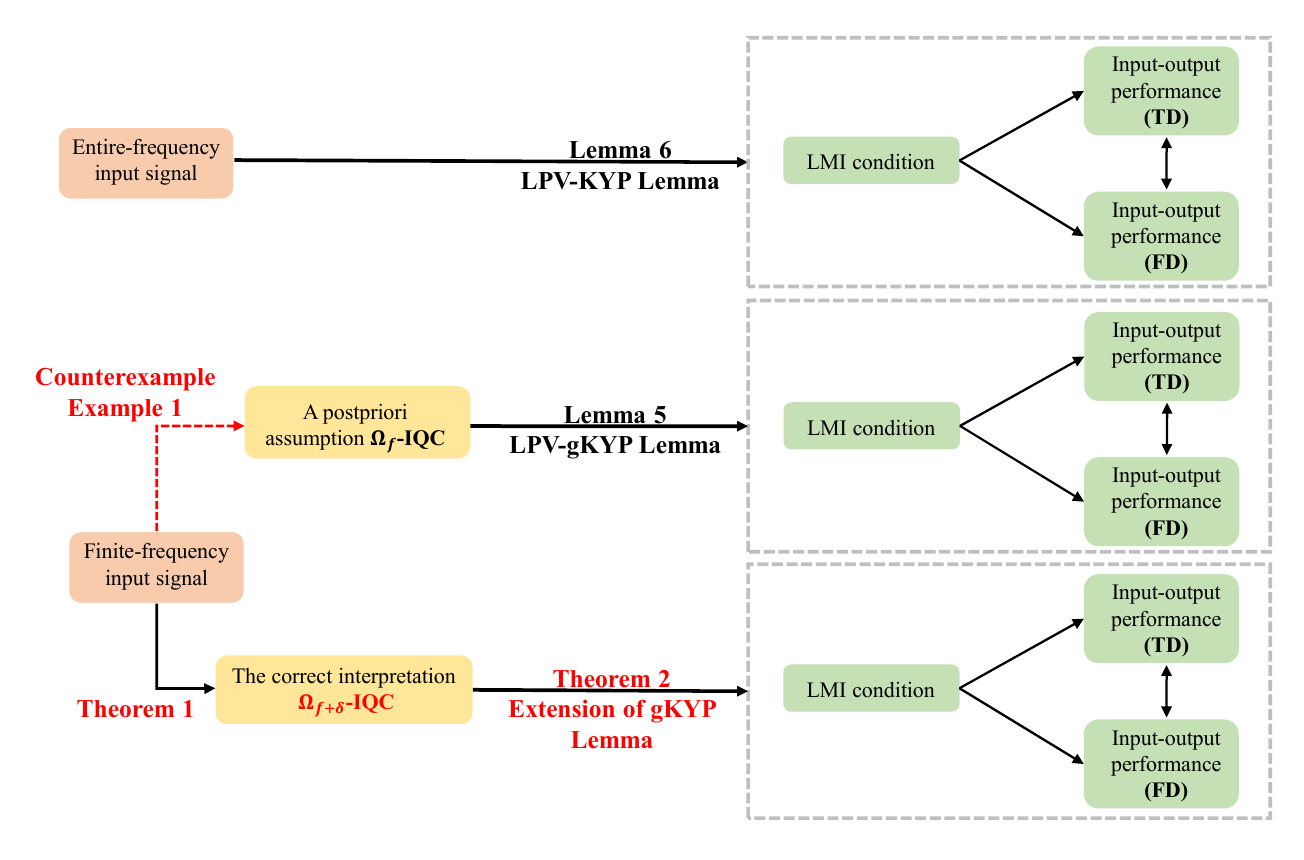}
    \caption{On finite-frequency analysis of general LPV systems.}
    \label{fig:IO-LPV-diagram}
\end{figure}

\end{remark}

\begin{remark}
The concrete procedures on how to converting infinite number LMIs to finite LMIs is omitted here, for more details, please refer to maximum at the vertices/partial convexity technique 
(\cite{cox2018affine}, Lem. 1).
Finally, the partial-convexity relaxation technical \cite{con_gahinet1996affine, con_chilali1999robust, con_mahmoud2002discrete} is adopted to generate numerical tractable algorithms, which allow one to check the time-varying parameter caused infinite-dimensional LMIs by a finite number of LMIs.

\end{remark}

\section{Numerical Experiments}

\begin{example}\label{example2}

Let us reconsider the LPV system given by Example \ref{exa:counterexample}. 
At first stage, 
let us check the gap between the matrix $A(\mathrm{p}(t))$ and finite-frequency range $\Omega_f:=[-1,1]$. 
Obviously, the uniform spectral radius of $A(\mathrm p(t))$ is greater than the spectrum of finite-frequency range $\Omega_f$, which means  
\[\begin{bmatrix}
     A(\mathrm p(t)) \\ I
       \end{bmatrix}^* (\Psi_{f}\otimes I) \begin{bmatrix}
     A(\mathrm p(t)) \\ I
         \end{bmatrix} < 0,\]             
so 
\[\bigl(d(A(\mathrm{p}(t)),\Omega_f)\bigl)^2=\lambda_{\rm max}\Bigl(-\left[\begin{smallmatrix}
     A^*(\mathrm{p}(t)) & I
       \end{smallmatrix}\right]
       (\Psi_{f} \otimes I)
       \left[\begin{smallmatrix}
     A(\mathrm{p}(t)) \\
     I
         \end{smallmatrix}\right] \Bigl)
         =164.62.\]

On the basis of this fact, we further compute the enlarged frequency interval ${\Omega}_{f+\delta}$ as follows:

Firstly, by simple computing in Matlab, 
the trace of finite-frequency controllability Gramian is
\begin{equation}
\tr\bigl(\mathbb W_{\mathrm p}(\Omega_f)\bigl)
= 0.4858.
\end{equation}

Then, by setting $c_1=0.5$, $c_2=0.6$, $c_3=7.4$ and solving two Lyapunov inequalities in \eqref{Lya_W_dp}, we obtain
\[
\tr\bigl(\mathbb{W}_{\dot{\mathrm{p}}}(\Omega_f)\bigl)
\leq 0.075351+ 0.026327   \approx0.1017.
\]

So, the minimum value of new frequency range satisfies
\begin{equation}
    \begin{array}{l}
    \delta^2= 
    \bigl(d(A(\mathrm p(t)),\Omega_f)\bigl)^2
    \tr\bigl(\mathbb W_{\dot {\mathrm p}}(\Omega_f)\bigl)
   \times
    \Bigl(  
    \tr(\mathbb W_{\mathrm p}(\Omega_f))
    \Bigr)^{-1}
    =34.4624.
    \end{array}
\end{equation}

Over the new frequency interval {${\Omega}_{f+\delta}:=[-5.955,5.955]$}, 
the updated minimal performance upper bound ${\gamma}_{f-u}^{\star}$ is given by:
\[\begin{array}{l}
\sqrt{\frac{\int_{0}^{\infty} y^*(\tau) y(\tau) \inva \tau}  {\int_{0}^{\infty} u^*(\tau) u(\tau) \inva \tau} }<{\gamma}_{f-{u}}^{\star} \approx 5.0313,
\end{array}\]
which, although larger than $\gamma_f^{\star}$, the actual input-output performance is no more than the bound $\gamma_{f-u}^{\star}$ (see Fig. \ref{fig:E3-gamma}). 
Moreover, compared with the bound $\gamma_e^{\star}\approx 5.2445$ obtained for the entire-frequency range, our given bound $\gamma_{f-u}^{\star}\approx 4.9161$ achieve an approximate 4.07\% improvement in performance.

\begin{figure}[htbp]
	\centering
	\includegraphics[width=0.7\textwidth]{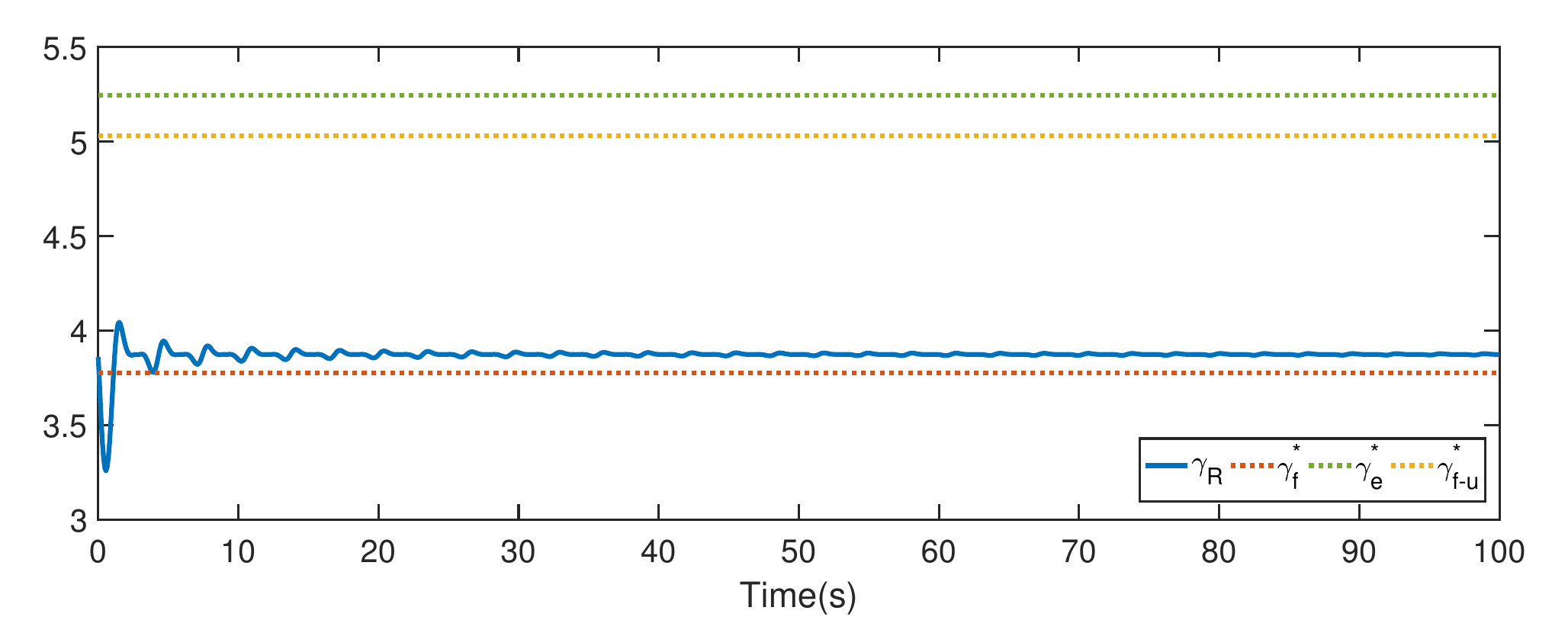}
	\caption{Actual input-output performance index $\gamma_R$, performance bounds $\gamma_f^{\star}$, $\gamma_{f-u}^{\star}$ and $\gamma_e^{\star}$. }
	\label{fig:E3-gamma}
\end{figure}

Additionally, the updated scalar-valued ${\Omega}_{f+\delta}$-IQC becomes non-negative as illustrated in Fig. \ref{fig:E3-S}.

\begin{figure}[htbp]
\centering
\includegraphics[width=0.7\linewidth]{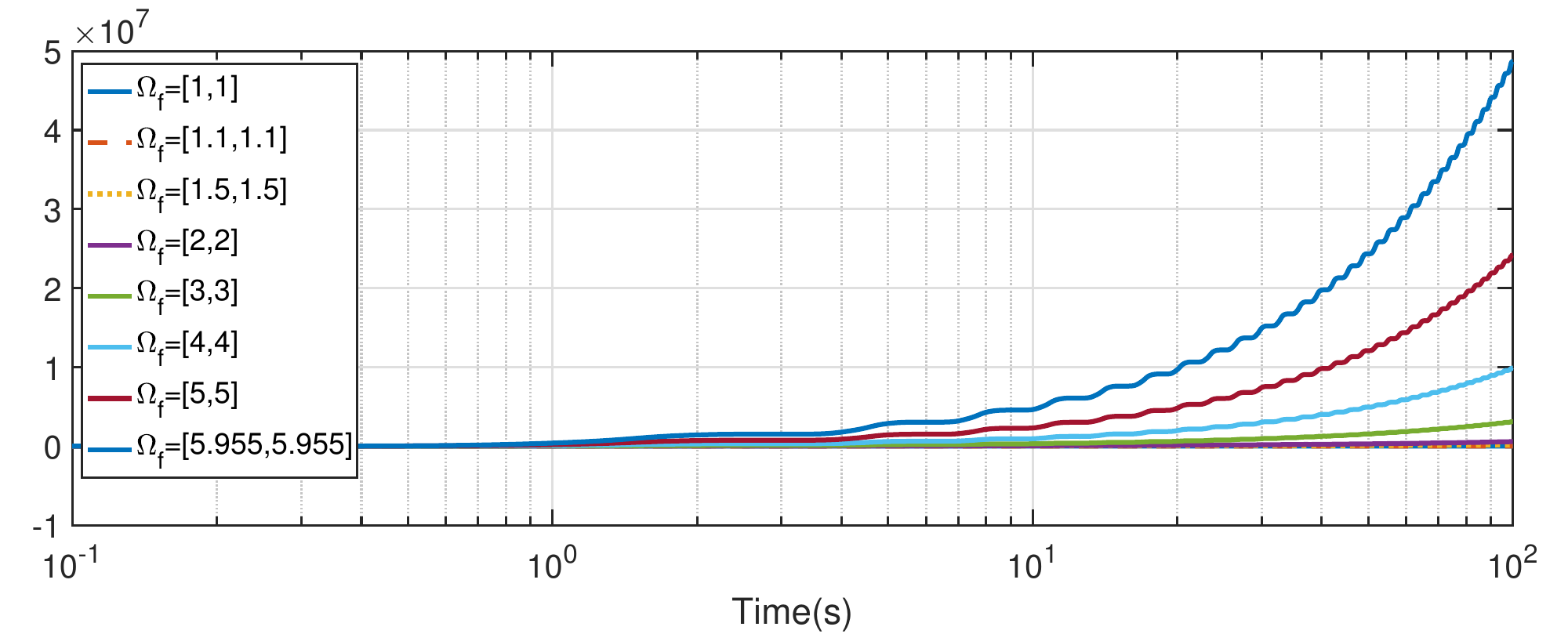}
\caption{The scalar-valued $\Omega_f$-IQCs with distinct finite-frequency ranges.}
\label{fig:E3-S}
\end{figure}

Naturally, if the chosen frequency range exceeds the uniform spectral radius of $A(\mathrm{p}(t))$, i.e., 
\begin{equation*}
    \varpi_l\geq \rho_{\rm unif}(A(\mathrm p(t)))=12.8306,
\end{equation*}
the gap between the matrix spectrum and the finite-frequency range vanishes. Consequently, no further enlargement of the frequency range is necessary.

Let us choose $\Omega_f:=[-12.8306,12.8306]$, then 
the updated minimal performance upper bound ${\gamma}_{f-new}^{\star}$ is given by:
\[\begin{array}{l}
\sqrt{\frac{\int_{0}^{\infty} y^*(\tau) y(\tau) \inva \tau}  {\int_{0}^{\infty} u^*(\tau) u(\tau) \inva \tau} }<{\gamma}_{f-new}^{\star} \approx 5.0313,
\end{array}\]
then the actual input-output performance does not exceed the updated minimal performance upper bound ${\gamma}_{f-new}^{\star}$, and the scalar valued $\Omega_f$-IQC is always non-negative, shown in Fig. \ref{fig:E4-gamma} and Fig. \ref{fig:E4-S}.

\begin{figure}[htbp]
\centering
\includegraphics[width=0.7\linewidth]{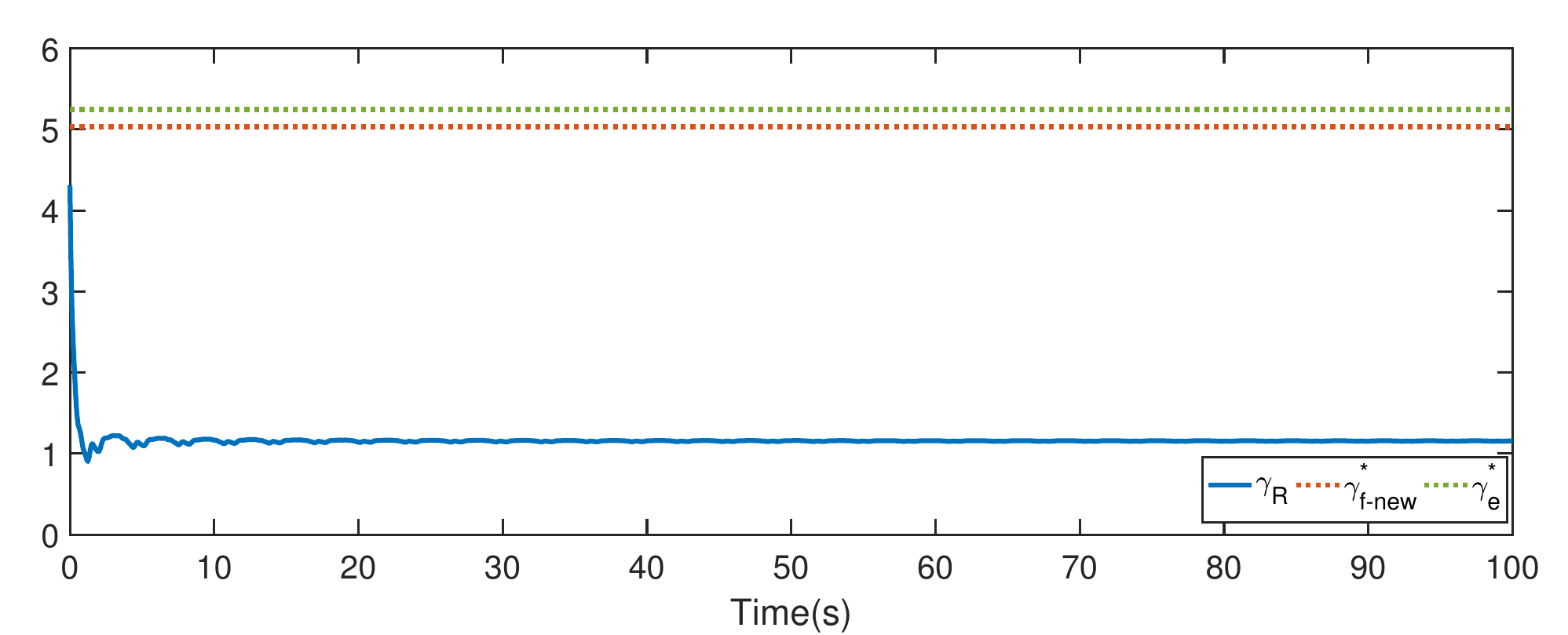}
\caption{Actual input-output performance index $\gamma_R$, performance bounds $\gamma_{f-new}^{\star}$ and $\gamma_e^{\star}$.}
\label{fig:E4-gamma}
\end{figure}

\begin{figure}[htbp]
\centering
\includegraphics[width=0.7\linewidth]{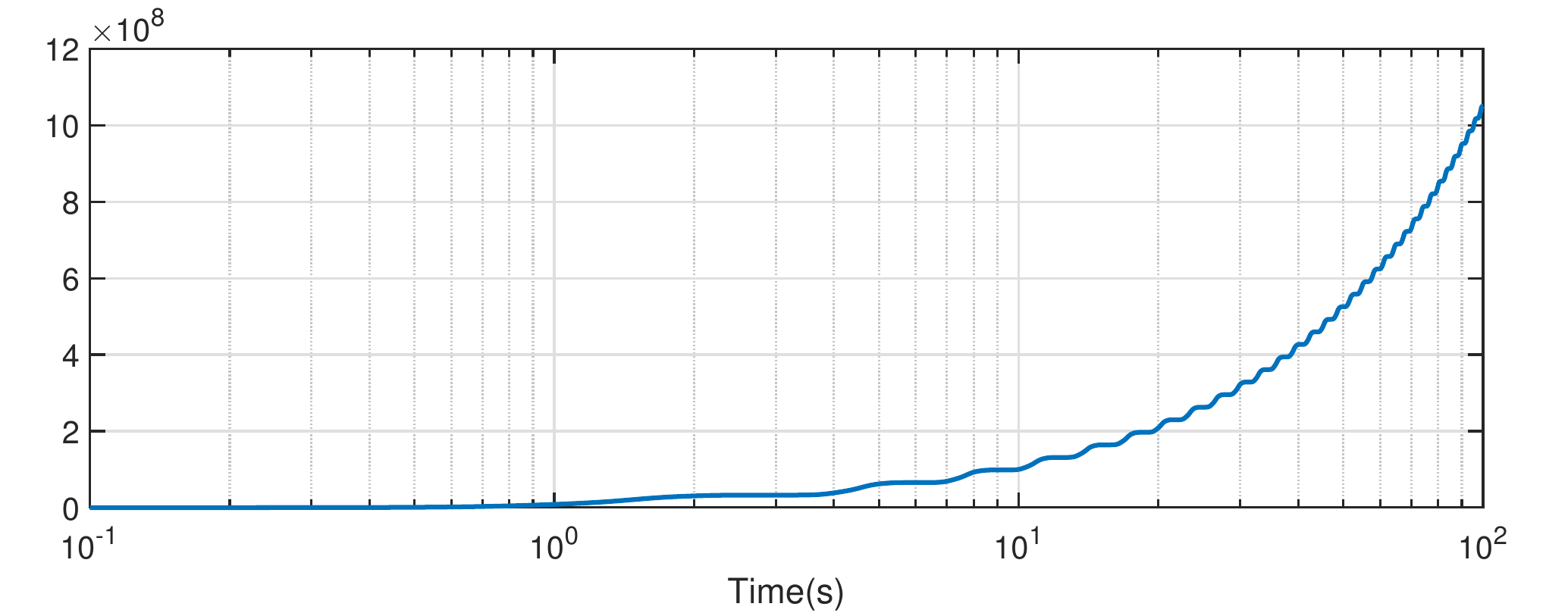}
\caption{The scalar-valued $\Omega_f$-IQC with finite-frequency range.}
\label{fig:E4-S}
\end{figure}

In the following Fig. \ref{fig:E4-Im-Re}, the connections among the system matrix's eigenvalues $\lambda(A(\mathrm p))$, finite frequency interval $\Omega_f:=[-\omega_l,\omega_l]$, enlarged finite frequency interval $\Omega_{f+\delta}:=[-\omega_{l+\delta},\omega_{l+\delta}]$, and the matrix uniform spectral radius $\rho_{\rm unif}(A(\mathrm p))$ is shown, and the blue solid dots represent the eigenvalues of the system matrix.

\begin{figure}[htbp]
\centering
\includegraphics[width=0.5\linewidth]{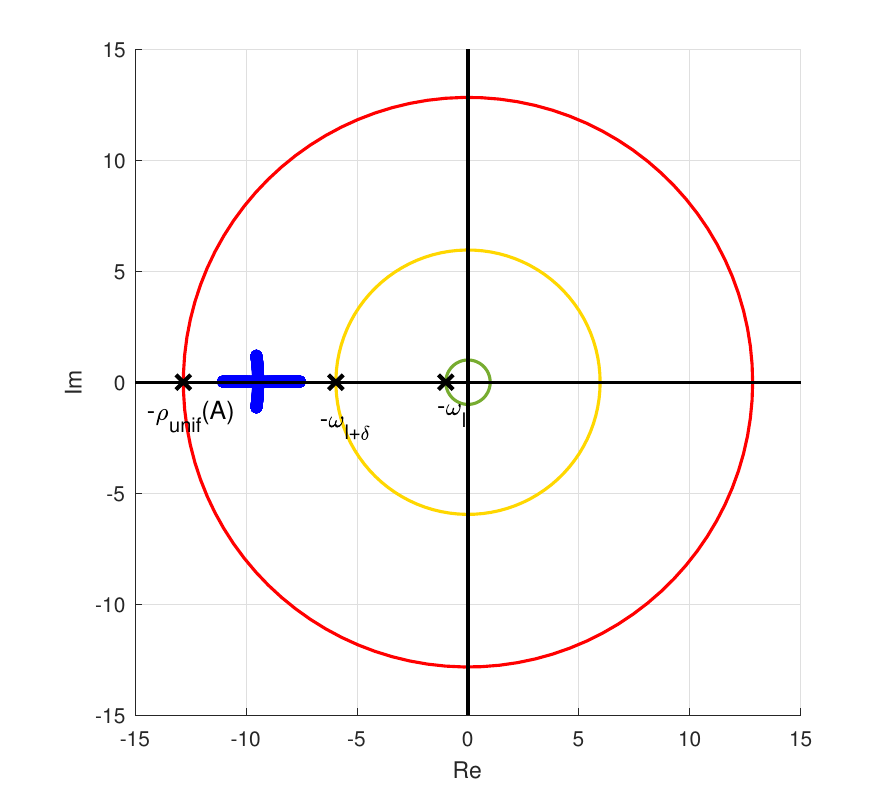}
\caption{The connections among the system matrix's eigenvalues $\lambda(A(\mathrm p))$, finite frequency interval $\Omega_f:=[-\omega_l,\omega_l]$, enlarged finite frequency interval $\Omega_{f+\delta}:=[-\omega_{l+\delta},\omega_{l+\delta}]$, and the matrix uniform spectral radius $\rho_{\rm unif}(A(\mathrm p))$.}
\label{fig:E4-Im-Re}
\end{figure}

The Fig. \ref{fig:E4-Im-Re} offers an intuitive understanding of the frequency range enlargement scheme. Firstly, it is shown that the eigenvalues of the system matrix $A(\mathrm p)$ are distributed outside the circle (green colored, radius$=\omega_l$), hence a frequency range enlargement in this case becomes a necessary step for performance analysis.  Also, it should be noted that the enlarged frequency range could be smaller than the uniform spectral radius of system matrix $A(\mathrm p)$, as the value of frequency enlargement is not only depend on the gap between the system matrix and finite-frequency range, but also depend on the trace of Gramian matrices (see formulas \eqref{delta}, \eqref{delta-UAS}).

\end{example}

\section{Conclusions and future work} \label{sec:conclusion}

A theoretically solid and practically applicable generalization of the celebrated gKYP lemma has been obtained for LPV systems by introducing an frequency range enlargement scheme in the construction of frequency-dependent IQC function. It is revealed that the effect caused by the intermodulation between input signals and the scheduling parameters can be captured by the geometrical gap of system poles and the given frequency range, as well as a group of controllability Gramians. Based on the discovery, a rigorously established gKYP-like condition is developed for LPV systems, which condition captures the input-output relationship of LPV systems under finite-frequency input signal assumption. The proposed extended gKYP lemma leads us to a fully reliable route for finite-frequency analysis of LPV systems. 

It is our belief that the methodology in our development can also be utilized for dealing with finite-frequency analysis and synthesis of various systems including time-varying terms or nonlinear terms, to name a few, linear time-delayed systems with varying delay.

\section*{Acknowledgments}

The corresponding author Xin Du would like to thank Prof. Tong Zhou from Tsinghua University and Prof. Guang-Hong Yang from Northeastern University for their thought-provoking questions about finite-frequency analysis of LPV systems on the occassion of Xin Du's Ph.D dissertation defense, which turned out to be the starting point for this 15 years lasting work.




\bibliographystyle{IEEEtran}
\bibliography{analysis_condition}

\begin{IEEEbiography}[{\includegraphics[width=1in,height=1.25in,clip,keepaspectratio]{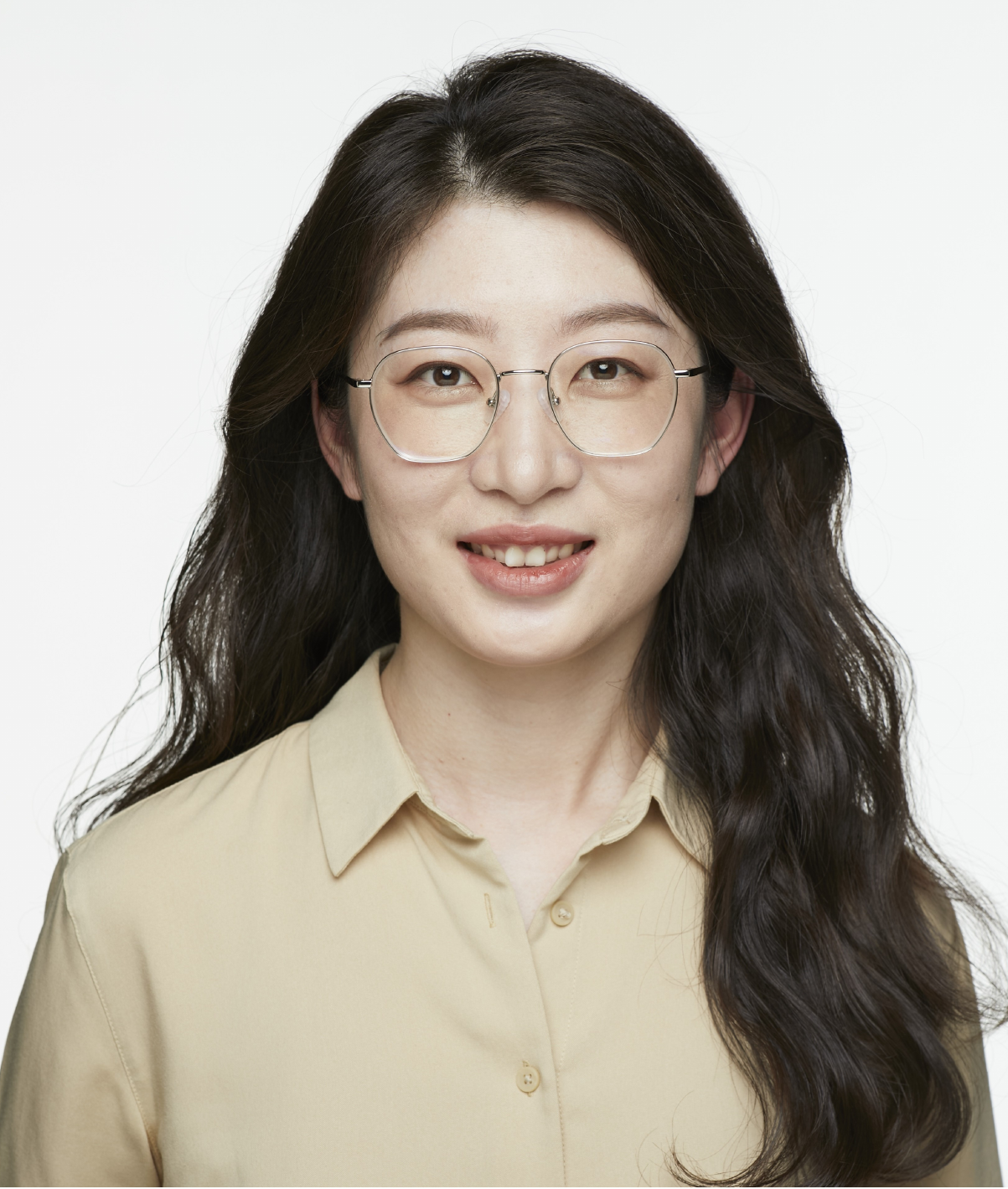}}]{Jingjing Zhang} received the B.S. degree in mathematics and applied mathematics from Hebei Normal University, Hebei, China, in 2017 and the M.S. degree in basic mathematics from Liaoning University, Liaoning, China, in 2020. She is currently pursuing the Ph.D. degree in control theory and control engineering at Shanghai University, Shanghai, China. 
	From 2021 to 2023, she was a joint Ph.D student in department computational methods in systems and control theory (CSC) at Max Planck Institute for Dynamics of Complex Technical Systems, Magdeburg, Germany. 
	From 2023 to now, she was a joint Ph.D student in State Key Laboratory of Mathematical Sciences, Academy of Mathematics and Systems Science (AMSS), Chinese Academy of Sciences (CAS), Beijing, China. Her research interest includes robust control, finite frequency switching control and adaptive control.
\end{IEEEbiography}
{\vspace{-4em}}

\begin{IEEEbiography}[{\includegraphics[width=1in,height=1.25in,clip,keepaspectratio]{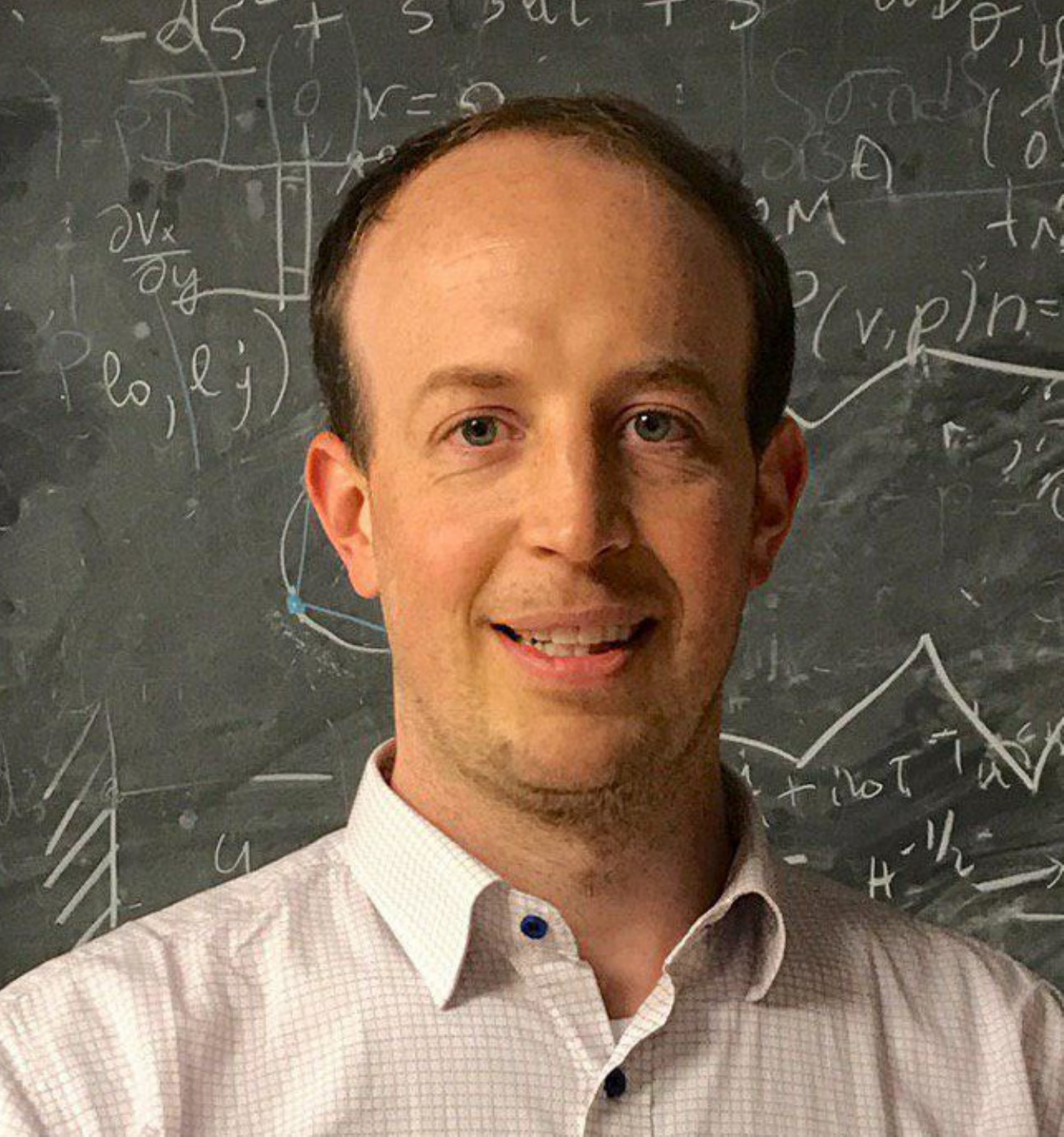}}]{Jan Heiland} received the Ph.D degree in mathematics from the Technical University of Berlin in 2013. Jan Heiland is a researcher and team leader at the Max Planck Institute for dynamics of complex technical systems in Magdeburg, Germany. He has gathered two years of practical experience in the private sector dealing with high-speed trains like the Talgo. Also, he has teaching duties at the University of Magdeburg. During February and March 2020 he collaborated as Visiting Professor at the ERC Advanced Grant project DyCon with Prof. Enrique Zuazua (FAU, University of Deusto and Universidad Autónoma de Madrid).
\end{IEEEbiography}

{\vspace{-4em}}
\begin{IEEEbiography}[{\includegraphics[width=1in,height=1.25in,clip,keepaspectratio]{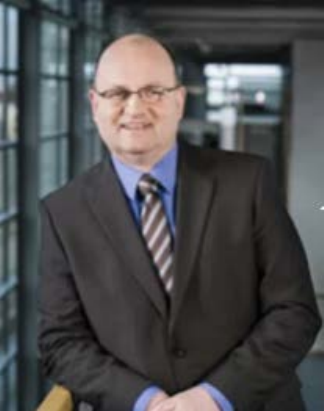}}]{Peter Benner} received the Diploma degree in mathematics from the RWTH Aachen University, Aachen, Germany, in 1993, the Ph.D. degree in mathematics from the University of Kansas, Lawrence, KS, USA, and the TU Chemnitz-Zwickau, Germany, in February 1997, and the Habilitation (Venia Legendi) degree in mathematics from the University of Bremen, Germany, in 2001. After spending a term as a Visiting Associate Professor with the TU Hamburg-Hamburg, Germany, he was a Lecturer in Mathematics with the TU Berlin, Germany, from 2001 to 2003. Since 2003, he has been a Professor of mathematics in industry and technology with the TU Chemnitz. In 2010, he was appointed as one of the four Directors of the Max Planck Institute for Dynamics of Complex Technical Systems, Magdeburg, Germany. Since 2011, he has been an Honorary Professor with the Otto-von-Guericke University of Magdeburg, Germany. His research interests include scientific computing, numerical mathematics, systems theory, and optimal control. He is a SIAM Fellow (Class of 2017).
\end{IEEEbiography}

{\vspace{-4em}}
\begin{IEEEbiography}[{\includegraphics[width=1in,height=1.25in,clip,keepaspectratio]{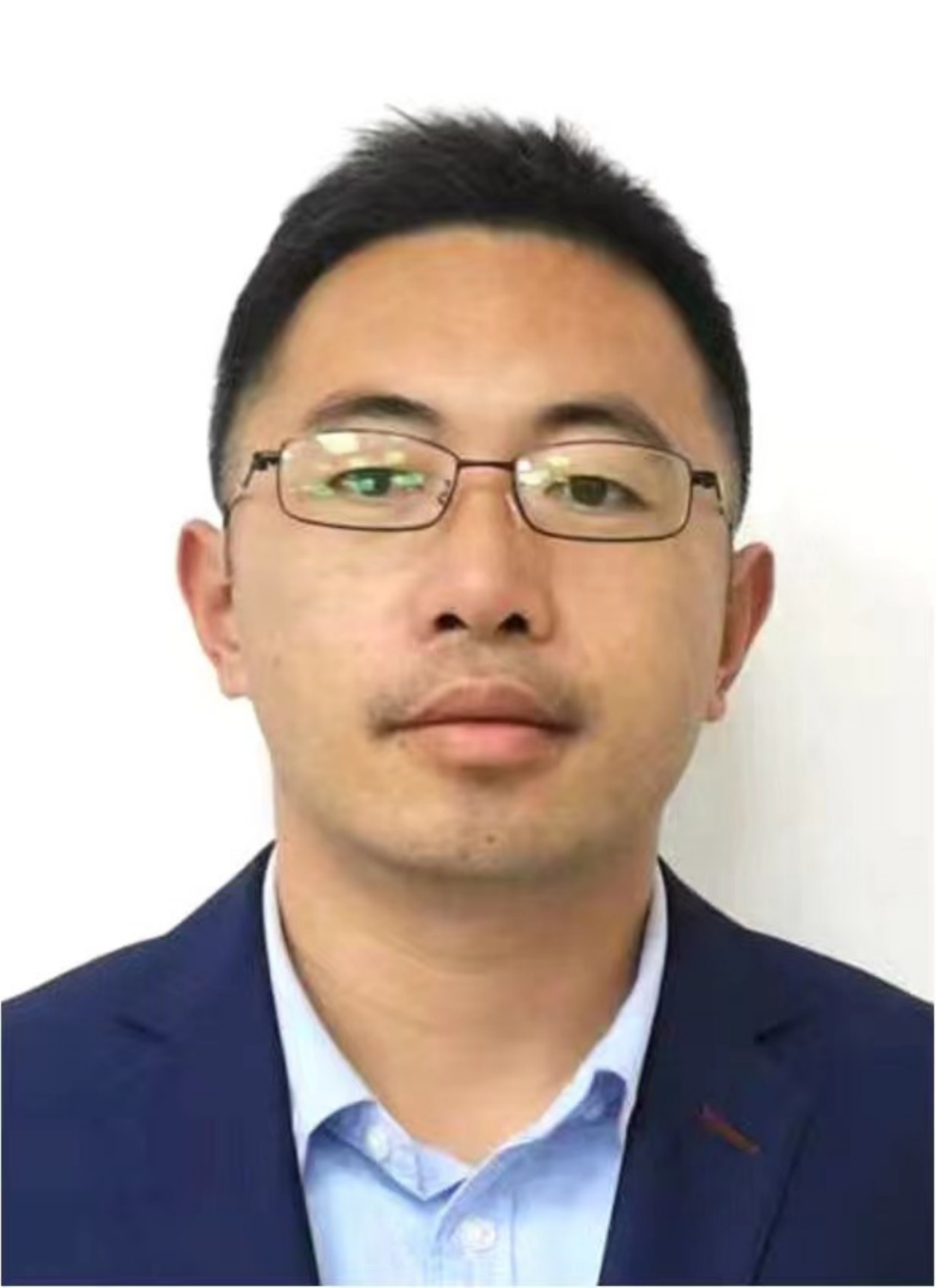}}]{Xin Du} received his B.S degree from University of Science and Technology Beijing, Beijing, China, in 2004. He received Ph.D. degree in Control Theory and Control Engineering from Northeastern University, Shenyang, China, in 2010.  During February 2010 to April 2013, He served as Lecturer/Associate Professor at Shanghai University. He has been a Post-doc research fellow at Max-Planck Institute for Dynamics of Complex Technical Systems, Magdeburg, Germany from April to August, 2012 and from December 2013 to August 2015. Currently, he is an program officer at governmental sector and do academic research individually and independently in his spare time. His research interests include modeling and control of complex systems, model order reduction, dynamical system and machine learning.  
\end{IEEEbiography}

\end{document}